\documentclass[a4paper,10pt,reqno]{amsart}
\usepackage{amsthm,amssymb,latexsym,amsmath}

\usepackage[latin1]{inputenc}
\usepackage{amsmath,amssymb,amsfonts,amsthm}

\newtheorem{theorem}{Theorem}[section]
\newtheorem{maintheorem}{Theorem}

\newtheorem{ThmSmall}{Corollary}

\newtheorem{Thm}{Theorem}

\newtheorem{nointthm}{Theorem}

\newcommand{\cod}{\text{cod}}

\newtheorem{proposition}[theorem]{Proposition}
\newtheorem{corollary}[theorem]{Corollary}

\theoremstyle{definition}
\newtheorem{definition}[theorem]{Definition}
\newtheorem{example}[theorem]{Example}
\newtheorem{problem}[theorem]{Problem}

\newtheorem{remark}[theorem]{Remark}
\numberwithin{equation}{section}
\usepackage{color}

\newcommand{\R}{\mbox{Res}}
\newcommand{\ar}{\textsf{arg}}

\newcommand{\F}{\mathcal F}
\newcommand{\K}{\mathcal K}
\newcommand{\Po}{\mathbb P}
\newcommand{\Om}{\Omega}
\newcommand{\C}{\mathbb C}
\newcommand{\Z}{\mathbb Z}

\title[Higher Codimensional Foliations and Kupka singularities]{Higher
  Codimensional Foliations with Kupka Singularities}
\date{\today}

\author[Calvo--Andrade, Corr\^ea, Fernandez--P\'erez]
{O. Calvo--Andrade \and M. Corr\^ea JR
\and A. Fern\' andez--P\' erez }

\address{\emph{O. Calvo--Andrade}: IMPA. Estrada Dona Castorina 110. Jardim Bot\^anico.}
\curraddr{CIMAT. Apdo. Postal 402, Guanajuato 36000. M\'exico.}
\email{omegar.mat@gmail.com}

\address{\emph{M. Corr\^ea Jr.}: Depto. de Mat.--ICEX
   Universidade Federal de Minas Gerais, UFMG}
\curraddr{Av. Ant\^onio Carlos 6627, 31270-901,
  Belo Horizonte-MG, Brasil.}
\email{mauriciomatufmg@gmail.com}

\address{\emph{A. Fern\'andez--P\'erez}: Depto. de Mat.--ICEX
  \\ Universidade Federal de Minas Gerais, UFMG}
\curraddr{Av. Ant\^onio Carlos 6627, 31270-901,
  Belo Horizonte-MG, Brasil.}
\email{arturofp@mat.ufmg.br}


\begin{document}

\begin{abstract}
We consider holomorphic foliations of dimension $k>1$ and
codimension $\geq 1$ in the projective space $\Po^n$, with a compact
connected component of the Kupka set. We prove that, if the
transversal type is linear with positive integers eigenvalues, then
the foliation consist on the fibers of a rational fibration
$\Phi:\Po^{n}\dasharrow \Po^{n-k}$. As a corollary, if $\mathcal{F}$
is a foliation such that $\dim(\F)\geq \cod(\F)+2$ and has
transversal type diagonal with different eigenvalues, then the Kupka
component $K$ is a complete intersection and we get the same
conclusion. The same conclusion holds if the Kupka set is a complete
intersection and has radial transversal type. Finally, as an
application, we find a normal form for non integrable codimension
one distributions on $\Po^{n}$.
\end{abstract}

\maketitle
\tableofcontents

\section{Introduction.}
\subsection{Notation}\label{Notation}

Let $M$ be a complex manifold. We denote by
$\mathcal{O}_M,\, \Theta_M$ and $\Om_M^p$ the sheaves of
holomorphic functions, holomorphic vector fields and holomorphic
$p$--forms over $M$. Therefore, if $n=\dim_{\C}M$, we denote by
$\mathcal{K}_M=\Om^n_M$ its canonical bundle
($\mathcal{O}_n,\Theta_n, \Om_n^p$ and $\mathcal{K}_n$ when
$M=\Po^n$). If $E$ is a holomorphic vector bundle, then
$\Om_M^p(E)$ denotes the sheaf of holomorphic $p$--forms with
values in $E$. No distinction will be
made between holomorphic vector bundles and locally
free analytic sheaves.

\subsection{Definitions}\label{defs}
Let $U\subset \C^n$ be an open set and $1\leq k<n$ an integer. A
holomorphic $(n-k)$--form $\Om\in \Om_{U}^{n-k}(U)$ is integrable, if for
each point $x\in U$, where
$\Om(x)\neq0$, there exits a neighborhood $V$ of $x$ and holomorphic 1--forms
$\eta_1,\dots ,\eta_{n-k}$ such that

\begin{itemize}
    \item $\Om|_V=\eta_1\wedge \cdots \wedge \eta_{n-k}$.
    \item There exists an $(n-k)\times (n-k)$
      matrix of holomorphic $1$--forms $\theta=(\theta_{ij})$ such that for all
      $i=1,\dots,n-k$ the following equality holds.
      \[ d\eta_i=\sum_{j=1}^{n-k} \theta_{ij}\wedge \eta_j.  \]
\end{itemize}

We recall the isomorphism between $\Theta_U\longleftrightarrow\Om^{n-1}_U$
defined by the contraction by a vector field.
\[
\frac{\partial}{\partial x_i}\longleftrightarrow \imath_{\frac{\partial}{\partial
x_i}}dx_1\wedge\dots\wedge dx_n= (-1)^i dx_1\wedge \cdots \wedge
\widehat{dx_i}\wedge\cdots \wedge dx_n
\]
therefore, $(n-1)$--forms in $\C^n$ are always integrable.

Let $M$ be a complex manifold of dimension $n$ and $ 1\leq k <n$.
A holomorphic foliation $\F$ on $M$, of dimension $\dim(\F)=k$ and
codimension $\cod(\F)=n-k$, may be defined by a triple
$\{\mathcal{U}=\{U_{\alpha}\},\Om_{\alpha},\lambda_{\alpha\beta}\}$
where
\begin{enumerate}
\item $\mathcal{U}=\{U_{\alpha}\}$ is an open covering of $M$.
\item $\{\Om_{\alpha}\}$ is a family of holomorphic, integrable $(n-k)$--forms, defined
  on the open set $U_{\alpha}$.
\item If $U_{\alpha}\cap U_{\beta}:=U_{\alpha\beta}\neq\emptyset$, we
  have the family
$\{\lambda_{\alpha\beta}\in \mathcal{O}^{\ast}(U_{\alpha\beta})\}$
satisfying
\[
\Om_{\alpha}|_{U_{\alpha\beta}}=\lambda_{\alpha\beta}\Om_{\beta}|_{U_{\alpha\beta}},
\]
\end{enumerate}
such that $\{\lambda_{\alpha\beta}\}\in\check{H}^1(\mathcal{U},\mathcal{O}^{\ast}).$

Let $\{W_{\sigma},\Om_{\sigma},\mu_{\sigma\tau}\}$ be another
triple satisfying conditions (1), (2) and (3) above. The triples
$\{U_{\alpha},\Om_{\alpha},\lambda_{\alpha\beta}\}\thicksim
\{W_{\sigma},\Om'_{\sigma},\mu_{\sigma\tau}\}$ are
\textit{equivalent} if there exists a pair $\{(V_i,\rho_i)\}$, where
$\{V_i\}$ is a refinement of the both coverings $\{W_{\sigma}\}$ and
$\{U_{\alpha}\}$, and $\{\rho_i\in\mathcal{O}^{\ast}(V_i)\}$ is a
family of functions such that if $V_i\subset U_{\alpha}\cap
W_{\tau}$ then
\[\Om_{\alpha}|_{V_i}=\rho_i\cdot\Om_{\tau}'|_{V_i},\]
it follows that $\lambda_{ij}=\rho_i\cdot \mu_{ij}\cdot \rho_j^{-1}$,
and then
$[\lambda_{\alpha\beta}]=[\mu_{\sigma\tau}]\in\check{H}^1(M,\mathcal{O}^{\ast})$.

Also, for two equivalent triples
$\{U_{\alpha},\Om_{\alpha},\lambda_{\alpha\beta}\} \thicksim
\{U_{\alpha},\Om'_{\alpha},\lambda_{\alpha\beta}\}$, the family of
functions $\{U_{\alpha},\rho_{\alpha}\}$ defines a global and
never vanishing holomorphic function on $M$.


We denote by $L$ the holomorphic line bundle over
$M$, represented by the 1--cocycle $\{\lambda_{\alpha\beta}\}\in
\check{H}^1(\mathcal{U},\mathcal{O}^{\ast})$. The family of
$(n-k)$-forms $\{\Om_{\alpha}\}$, glue in order to obtain a global
holomorphic section $\{\Om_{\alpha}\}\in
\check{H}^0(\mathcal{U},\Om_M^{n-k}(L))$. This motivates the
following definition.

\begin{definition}\label{foliation}
Let $M$ be a complex manifold of dimension $n$. Set $1\leq k<n$. A
holomorphic foliation $\F$ on $M$ of dimension $\dim({\F})=k$ and
codimension $\cod({\F})=n-k$, is an equivalence class of
integrable sections $[\Om]$, where $\Om \in H^0(M,\Om^{n-k}(L))$,
and $L$ is a holomorphic line bundle over $M$.
\end{definition}

The singular set of the foliation represented by the section $\Om$, is
defined by
\[
 S(\F)=S(\Om):=\{p\in M | \Om(p)=0\},
\]
it is an analytic set, with several irreducible components. We are
going to denote by $S_r(\F)$ the union of irreducible components of
dimension $r$. We may assume always that the maximal dimension of the
irreducible components of singular set is $n-2$.

We denote by $\F_k(M,L)\subset H^0(M,\Om_M^{n-k}(L))$, the set of
equivalence classes integrable sections or foliations on $M$, in
the case of the projective space, we use the notation
$\F_k(n,c)$ instead $\F_k(\Po^n,\mathcal{O}_n(c))$. It is well
known that, if $M$ is compact, $\F_k(M,L)$ is a projective variety
in general, with several irreducible components.

Let $\F\in\F_k(M,L)$ be a foliation represented by the section
$\Om$. The \textit{tangent sheaf} $T_{\F}\hookrightarrow \Theta$,
is defined as follows: the stalk at any point $x\in M$ is
\[
(T_{\F})_x=\{ \textbf{X}\in \Theta_x|\imath_{\mathbf{X}}\Om=0\},
\]
the quotient $N_{\F}:=\Theta/T_{\F}$ is the \emph{normal sheaf} of
$\F$. Then, these sheaves are defined by the exact sequence
\begin{equation}\label{tangent}
0\rightarrow T_{\F}\rightarrow \Theta \rightarrow
N_{\F}\rightarrow 0.
\end{equation}

We have that $T_{\F}$ is a coherent subsheaf of $\mathcal{O}_M$ modules,
that is closed under the Lie bracket of vector fields. In some
cases, it is locally free. The normal sheaf is also coherent, but
it is not locally free over the singular set. The sequence (\ref{tangent})
gives an alternative definition of holomorphic foliation \cite{baum}.

The dual of $T_{\F}$ is \textit{the cotangent sheaf} of $\F$. It
will be denoted by $T^{*}\F$. Its determinant, i.e.
$\det(T^{\ast}\F)=(\bigwedge^{k}T^{*}\F)^{**}$ is  the
\textit{canonical bundle of the foliation}, it will be denoted by
$\mathcal{K}_{\F}=\Om^k_{\F}$. Over the smooth locus $M\setminus
S(\F)$, the sheaves of the sequence (\ref{tangent}) are locally
free sheaves. Since the singular set has codimension at least two,
we obtain the adjunction formula
\[
\mathcal{K}_M=\mathcal{K}_{\F}\otimes \det
N^{*}_{\F}=\mathcal{K}_{\F}\otimes L^{\ast}\quad\mbox{Then}\quad
\mathcal{K}_{\F}=(L\otimes\mathcal{K}_M).
\]

We see that the holomorphic line bundle $L$ correspond to the
determinant line bundle of the normal sheaf $\det(N_{\F})$.

After the dimension and the codimension of a holomorphic foliation,
the main discrete invariant, is the
\textit{first Chern class of the normal sheaf} $N_{\F}$, and denoted by $c_1(N_{\F})$.
We define
\[
\F\in\F_k(M,L)\quad \mbox{ then } \quad c_1(N_{\F}):=c_1(L)\in H^2(M,\Z)
\]

For foliations on the projective space, we also define the
\textit{degree of} $\F$, it is the number
$\deg(\F)=c_1(N_{\F})-{\cod}(\F)-1$.

\begin{definition}\label{factor}
Let $\Om\in H^0(M,\Om^{n-k}(L))$ be an integrable $(n-k)$--form.
\begin{enumerate}
\item A
holomorphic (meromorphic) \textit{integrating factor} is a
holomorphic (meromorphic) section $\varphi$ of the line bundle $L$
such that
\[\eta= \frac{\Om}{\varphi}\]
is an integrable closed meromorphic form.

\item A meromorphic function $f:M\dasharrow \Po^1$ is a \textit{meromorphic first integral} if
\[ df \wedge \Om=0\]
\end{enumerate}
\end{definition}

Observe that if $\varphi$ and $\psi$ are two linearly independent integrating
factors of the foliation represented by $\Om$, then the rational function $f=\varphi/\psi$
is a meromorphic first integral of $\Om$.

\begin{definition}
The Kupka singular set of an integrable $(n-k)$--form $\Om$ is defined by
\[
 K(\F)=K(\Om):=\{p\in M | \Om(p)=0, d\Om(p)\neq 0\}\subset S(\F).
\]
\end{definition}

The singular and the Kupka sets, are well defined, they don't
depend on the $(n-k)$--form
$\Om$ representing the foliation.

The Kupka set has the \textit{local product property} (see Theorem
\ref{mede}): Given a foliation $\F\in\F_k(M,L),$ and a connected
component $K\subset K(\F)\subset M$ of its Kupka set, there exists
a unique germ of a foliation at $0\in \C^{n-k+1}$, represented by
a $(n-k)$--form $\eta\in \Om^{n-k}_0(\C^{n-k+1})$ or equivalently
a germ of a holomorphic vector field $\mathbf{X}\in
\Theta_0(\C^{n-k+1})$, and related by the formula
\[
\eta=\imath_{\mathbf{X}}dx_0\wedge\cdots \wedge dx_{n-k},
\]
such that for each point $x\in K$, there exists a pair $(U,\phi)$ where
$U\subset M$ is an open set and $\phi:(U,x)\to(\C^{n-k+1},0)$ is a
submersion such that
\begin{enumerate}
  \item $\phi^{-1}(0)=K\cap U$, therefore $K$ is a smooth submanifold and
    $\mbox{dim}(K)=\mbox{dim}(\F)-1$.
  \item The $(n-k)$--form in $\phi^{*}\eta$ defines $\F|_{U}$.
  \item The Kupka set is persistent under deformations of the foliation.
\end{enumerate}

Locally, the foliation looks like $\mathcal{G}_{\mathbf{X}}\times
\C^{k-1}$, the product of a singular foliation by curves in
$\C^{n-k+1}$ defined by the solutions of $\mathbf{X}$ times a non
singular holomorphic foliation of dimension $k-1$.

The set
\[
\K_k(M,L)=\{ \F\in\F_k(M,L)| K(\F) \supset K \mbox{  a compact
  connected component}\},
\]
is an open subset of $\F_k(M,L)$ and hence, its closure is a union
of irreducible components.

An important problem in the theory of foliations, is the
classification of the foliations $\K_k(M,L)\subset\F_k(M,L),\quad
n\geq3$. The codimension--one case i.e., when $1<\dim(\F)=n-1$,
for foliations on the projective space, was first considered by
Cerveau and Lins Neto in \cite{cerveau}. They proved the following
Theorem:

\begin{theorem}
\label{CL}
Let $\F\in \K_{n-1}(n,c),\quad n\geq3$. If $K(\F)$ is a complete intersection
then the leaves of $\F$ are the fibers of a rational map
$\Phi:\Po^n\dasharrow\Po^1.$
\end{theorem}

Furthermore, Cerveau and Lins Neto conjectured that $K(\F)$
is always a complete intersection. Recently, Brunella \cite{brunella}
and Calvo-Andrade \cite{joseom, omegar}, have given a positive answer
about this conjecture. As consequence of these results, there is a complete
classification of these foliations on the projective space
$\Po^n,\quad n\geq3$.

\subsection{Statement of the results}\label{results}
In this paper, we study the set of foliations $\K_k(M,L)\subset
\F_k(M,L)$ of holomorphic foliations of dimension $1<k\leq n-1$
with a compact connected Kupka component in its singular set.
We mainly assume that $1<k\leq n-2$.

Given a foliation $\F\in\K_k(M,L)$ with Kupka set $K=K(\F)$. It is natural
to ask about the relations between the invariants of the embedding
$\jmath:K\hookrightarrow M$, mainly,
the total Chern class of the normal bundle $c(\nu(K,M))\in
H^{\ast}(K,\mathbb{Z})$, and the invariants of the foliation.

For instance, Proposition \ref{subcanonical} said that
$\wedge^{k+1}\nu(K,M)=L|_K$, and in particular
\[
c_1(\nu(K,M))=j^{\ast}c_1(N_{\F})\in H^2(K,\mathbb{Z}),
\]
which implies that, if $L$ is an ample line bundle, the embedding
$\jmath:K\hookrightarrow M$ is not topologically trivial and at
least, the first Chern class of the normal bundle, is the
restriction of a non zero class on $H^2(M,\Z)$.

On the other hand, since the transversal vector field $\textbf{X}$ has
$Div(\mathbf{X})_0\neq0$, it has not trivial linear part
$\mathbf{X}_1$. The Jordan decomposition of $\mathbf{X}_1$,
implies that the normal bundle of a Kupka component, decomposes in
a direct sum of subbundles, or it is indecomposable and
projectively flat (Theorem \ref{normal}). As we will see later, this
property, imposes serious restrictions on the geometry of the
embedding of Kupka set.

In section \ref{sec:BB}, we prove a deeper relation between the
transversal vector field, the first Chern class of the normal
sheaf of the foliation and the fundamental class of the Kupka set
$\jmath(K)\subset M$.

The Theorem \ref{degK} of section \ref{sec:BB}, states that, for a foliation $\F\in
\K_k(M,L)$ with a compact connected component $K$ and transversal
type $\mathbf{X}$ the fundamental class of $K\subset M$ satisfies
the formula
\[
[K]\cdot \R(c_1(J_{\mathbf{X}})^{\cod(\F)+1},\mathbf{X},0)=
c_1(N_{\F})^{\cod(\F)+1}.
\]

In the equation above,
$\R(c_1(J_{\mathbf{X}})^{\cod(\F)+1},\mathbf{X},0)$ is the
Baum--Bott residue of the transversal vector field and $[K]$
denotes the fundamental class of $K\subset M$.

Let $\F$ be a foliation with a Kupka component with linear
transversal type
\[
\mathbf{X}(x_0,\dots,x_{n-k})=\sum_{i=0}^{n-k}\lambda_i
x_i\frac{\partial}{\partial x_i},\quad \mbox{with}\quad
\lambda_i\in\mathbb{N}\quad  \lambda_i\neq\lambda_j.
\]

In section \ref{sec:Kupka}, we introduce the notion
of \textit{resonance}. If the transversal vector field is linear,
and diagonal with positive integers eigenvalues, we find a
meromorphic integrating factor $G$ for the transversal model
$\Om=\imath_{\mathbf{X}}dx_0\wedge\dots \wedge dx_{n-k}$.
We get the normal
form:

\[
\frac{\Om}{G}= \left(\sum_{j=0}^{\ell}(-1)^i\lambda_i
\frac{dx_0}{x_0}\wedge\dots \widehat{\frac{dx_i}{x_i}}\cdots\wedge
\frac{dx_{\ell}}{x_{\ell}}\right)\wedge
d\varphi_{\ell+1}\wedge\cdots\wedge
d\varphi_{n-k},
\]
and a meromorphic function $F:U\dashrightarrow \Po^{n-k},$ whose fibers are the leaves.

The subset $\Lambda_{NR}=\{\lambda_0<\dots<\lambda_{\ell}\}$ is
the maximal subset of non--resonant eigenvalues. On the other hand,
for all $j=1,\dots,n-k-\ell$, the set of eigenvalues
$\Lambda_{R}=\{\lambda_{\ell+1}<\dots <\lambda_{n-k} \}$, we may
find resonance relations of the type
\[
\lambda_{\ell+j}=\lambda_0\cdot m^{\ell+j}_0+\cdots
+\lambda_{\ell}\cdot m^{\ell+j}_{\ell}.
\]

We define the polynomial functions
\[
h^{\ell+j}(x_0,\dots,x_{\ell})=x_0^{m_0^{\ell+j}}\cdots
x_{\ell}^{m^{\ell+j}_{\ell}},\quad
\varphi_{\ell+j}=\frac{x_{\ell+j}}{h^{\ell+j}(x_0,\dots
,x_{\ell})}.
\]

With a more carefull work, the above formal may be written, on a neighborhood of the Kupka set as
\begin{equation}\label{FNVK}
\frac{\Om}{\widetilde{G}}= \left(\sum_{j=0}^{\ell}(-1)^i\lambda_i
\frac{dx_0}{x_0}\wedge\dots \widehat{\frac{dx_i}{x_i}}\cdots\wedge
\frac{dx_{\ell}}{x_{\ell}}\right)
\wedge \theta_{\ell+1}\wedge \cdots\wedge \theta_{n-k},\qquad d\theta_j=0.
\end{equation}

With the use of (\ref{FNVK}), we prove the
following generalization of the Theorem (\cite[Theorem
B]{cerveau}).

\begin{maintheorem}\label{maintheo}
Let $\F\in \K_k(n,c)$ be a holomorphic foliation with
$S_{k-1}(\F)=K(\F)$ and transversal type
\[
\mathbf{X}=\sum^{n-k}_{i=0}\lambda_{i}x_{i}\frac{\partial}{\partial{x}_{i}},
\quad\lambda_i\in\mathbb{N}\quad
\lambda_{i}\neq \lambda_j, \mbox{ for all } i\neq j.
\]
Then the leaves of the foliation are the fibers of a rational
fibration and it is represented by a $(n-k)$--form of the type
\[
\Om(z_0,\dots,z_n)=\sum^{n-k}_{i=0}(-1)^i \lambda_{i}\cdot
f_{i}(z_0,\dots,z_n) df_{0}
\wedge\ldots\wedge\widehat{df_{i}}\wedge\ldots\wedge df_{n-k}.
\]
\end{maintheorem}

We say that a foliation $\F$ has \textit{small codimension} when
\begin{equation}
  \label{eq:dimension}
   \dim(\F)\geq \cod(\F)+2.
\end{equation}

We will see that this condition, we have strong restrictions on
the geometry of the Kupka set and the possibilities of the
transversal type.

Let $\F\in\K_k(n,c)$ be a foliation with small codimension, by Bezout's theorem, $S_{k-1}(\F)=K(\F)=K$
and it is compact and connected. Also, if the transversal type is linear with different eigenvalues and
$n\geq5$, according with \cite{Bad}, the inclusion map induces an isomorphism
$\jmath^{\ast}:\mbox{Pic}(\Po^n)\stackrel{\thicksim}{\rightarrow}
\mbox{Pic}(K) \simeq \mathbb{Z}$. We obtain the following result.

\begin{ThmSmall}\label{smallcod}
Let $\F\in \K_k(n,c)$ be a foliation of small codimension and
linear transversal type
\[
\mathbf{X}=\sum_{i=0}^k \lambda_i x_i \frac{\partial}{\partial
x_i}\qquad (\lambda_i- \lambda_j)\neq 0\neq \lambda_j\quad \mbox{for all}\quad i\neq
j,
\]
then $\lambda_i\in \mathbb{N}$ and $\F$ is a rational fibration.
\end{ThmSmall}

We also have a version for radial transversal type (\ref{RadLoc}).

\begin{Thm}\label{Radial}
Let $\F\in \K_k(n,c)$ with a Kupka set $K$ and radial transversal
type. If $K$ is a complete intersection, then $\F$ is a rational
fibration of the type
\[[f_0:\dots:f_{n-k}]:\Po^n\to\Po^{n-k},\quad deg(f_j)=\frac{c}{n-k+1}\]
\end{Thm}
\vglue 10pt

Finally, in section (\ref{sec:Noint}), we have an application to
codimension one non integrable distributions $\mathcal{D}$.

Let $\omega$ be a germ of  $1$-form on $(\mathbb{C}^n,0)$.  We
denote by
\[ (d\omega)^s= \overbrace{d\omega \wedge \cdots \wedge d\omega}^{\mbox{s
times}}.
\]

Let $\omega\in H^0(M, \Om_M^1(L))$. We define the \emph{class of the
foliation} induced by $\omega$ to be the integer $r$ for which
generically
\[
\omega\wedge (d\omega)^{r-1}\neq 0, \quad  \omega\wedge
(d\omega)^{r}\equiv 0.
\]

The Kupka set of the distribution $\mathcal{D}$ induced by $\omega\in H^0(M,
\Om_M^1(L))$ is defined by
\[
K(\mathcal{D})=\{p\in M;\ \omega(p)=0, \ (d\omega)^{r}(p)\neq 0 \}\subset Sing(\mathcal{D}).
\]

\begin{nointthm}
Let $\mathcal{D}$ be a codimension one distribution on $\Po^n$ of class $r$,
given by a $1$--form $\omega \in H^0(\Po^n, \Om_{\Po^n}^1(d+2))$
such that $\cod Sing( (d\omega)^r)\ge 3$. Suppose that the Kupka
component of $\F$   has transversal type
\[
\sum_{i=0}^{r-1}(x_idx_{i+r}- x_{i+r}dx_{i}).
\]

Then $\mathcal{D}$ is induced, in homogenous coordinates, by the $1$-form
\[
\sum_{i=0}^{r-1}(f_idf_{i+r}- f_{i+r}df_{i}).
\]
\end{nointthm}

\vglue 10pt
\section{The Kupka set.}
\label{sec:Kupka}

In this section, we  study the main properties of the
Kupka set. First, we study the behavior on a neighborhood of the Kupka set.
As we will see, the geometric properties of the Kupka
set, implies a strong rigidity on the global behavior of the
foliation.

\subsection{Basic properties of the Kupka set}

\begin{definition} Let $\Om\in H^0(M,\Om^{n-k}(L))$ be an
integrable section. The Kupka singular set is defined by
\[
 K(\F)=K(\Om):=\{p\in M\,|\,\Om(p)=0, d\Om(p)\neq 0\}.
\]
\end{definition}

The Kupka set is well defined. Let $\Om$ an integrable $(n-k)$
form andy let $\Om'$ be another integrable form which define the
same foliation, then
\[
   \Om=h\cdot \Om',\quad h\in\mathcal{O}^{\ast}
\]
then
$\Om(x)=h(x)\cdot \Om'(x)=0$  if and only if
$\Om(x)=\Om'(x)=0,$ and for any $x\in S(\Om)=S(\Om')$ we have $d\Om(x)=h(x)\cdot d\Om'(x)$.
Therefore $d\Om(x)\neq0$ if and only if $d\Om'(x)\neq0$.

Now, we begin with the local product structure.

\begin{theorem}
[Medeiros \cite{medeiros}]\label{mede}
Let $\F\in\F_k(M,L)$ be a $k$--dimensional holomorphic foliation on $M$ and let
$K\subset K(\F)$ be a connected component. Then there exists
a germ of a holomorphic vector field $\mathbf{X}\in
\Theta_0(\C^{n-k+1})$ called the \textit{transversal type}
and open covering by charts $\{(U_\alpha,\phi_\alpha)\}$ of a neighborhood
of $K$ such that
\begin{enumerate}
\item $K(\F)$ is a $k-1$--dimensional submanifold of $M$. Namely
\begin{eqnarray*}
  \phi_{\alpha}: U_\alpha& \rightarrow &\C^{n-k+1}\times\C^{k-1} \\
 u & \mapsto &(x_{\alpha}(u),y_{\alpha}(u)),\qquad U_\alpha\cap
 K=\phi_\alpha^{-1}(0,y_\alpha)
\end{eqnarray*}
\item The $(n-k)$--form
\[\Om_{\alpha}=\phi^{*}_{\alpha}(\imath_{\mathbf{X}}dx_{0}\wedge\ldots\wedge
dx_{n-k}).\]
represents the foliation on $U_{\alpha}$.
\item $K(\F)$ is persistent under
variation of $\Om$; namely, for $p\in K(\F)$ with transversal type
$\eta$ andy for any foliation $\F^{\prime}$ sufficiently
close to $\F$, there is a holomorphic p--form $\eta^{\prime}$, close
to $\eta$ and defined on a neighborhood of $0\in\C^{n-k+1}$ and a
submersion $\phi^{\prime}$ close to $\phi$, such that $\F^{\prime}$ is
defined by $(\phi^{\prime})^{\ast}\eta^{\prime}$ on a neighborhood of
$p$.
\end{enumerate}
\end{theorem}

\begin{remark}
Let us explain the point (3) of the Theorem (\ref{mede}) for the
families $\K_k(M,L)$ on a compact complex manifold $M$.

Let $\{\F_t\}_{t\in\mathcal{T}} $ a family of foliations such that
$\F_0\in \K_k(M,L)$. Then as a consequence of (3), we get a family
of compact submanifolds $\jmath_t:K_t\hookrightarrow M$ such that
$\jmath(K_t)=K(\F_t)$ and having transversal type $\mathbf{X}_t$.

The possible transversal type, is a germ at zero of holomorphic vector field,
it belongs to a space of infinite dimension, but $\K_k(M,L)$ has
finite dimension.

For instance, for the rational components
$\mathcal{R}(n:d_0,\dots,d_{n-k})$ defined in (\ref{ex:1}), the transversal type
remain fixed for any family $\F_t\in \mathcal{R}(n:d_0,\dots,d_{n-k})$
\[ \mathbf{X}_t=\sum_{i=0}^{n-k}d_j x_j\frac{\partial}{\partial x_j}.\]
\end{remark}

Given a foliation $\F$ represented by a section $\Om \in
H^0(M,\Om^{n-k}(L))$. Recall that the stalk at $x\in M$ of the
\textit{tangent sheaf} $(T_{\F})_x\subset\Theta_x$,
consist of the sheaf of the germs of holomorphic vector fields that vanishes
the form $\Om$.
\[ (T_{\F})_x=\{ \mathbf{Y}\in\Theta_x| \imath_{\mathbf{Y}}\Om=0 \}.\]

It follows that there is a coordinate system around each point of
the Kupka set
$(x,y)\in\C^{n-k+1}\times\C^{k-1}\simeq\C^n$, the form $(n-k)$--form
\[
\Om=\Om(x,y)=\imath_{\textbf{X}}dx_0\wedge \dots\wedge dx_{n-k}
\]
defines the  foliation. Then, the holomorphic vector fields
\[
\left\{
{\mathbf{X}}(x),\frac{\partial}{\partial
  y_1},\dots,\frac{\partial}{\partial y_{k-1}}
\right\}
\]
generate the tangent sheaf of the foliation as a $\mathcal{O}$ module
around $x\in K$, therefore, it is locally free.

\subsection{Geometry of the Kupka set.} We assume
that $K\subset K(\F)$ is a compact and connected component

We say that a compact (not necessarily connected) submanifold
$X\hookrightarrow M$ is \textit{subcanonicaly embedded}, if its
canonical bundle extends to a line bundle defined on $M$. That is
\[
\mathcal{K}_X=L|_X\qquad \mbox{for some}\qquad L\in Pic(M)
\]

Let $\F$ be a foliation with a compact Kupka component $K$. Consider
now the family of $(n-k)$--forms
$\{\Om_{\alpha}=\phi_{\alpha}^{\ast}\Om\}$ as in the theorem
(\ref{mede}). Since $d\Om(p)\neq0$ at $p\in K(\Om)$, we have
\[
  d\Om_\alpha(p)  =
  \phi_{\alpha}^{\ast}d\Om(0,y_{\alpha})=Div(\mathbf{X})(0)\,dx_{\alpha}^{0}\wedge
  \cdots \wedge dx_{\alpha}^{n-k}\neq 0,
\]
 where
\[
 Div(\mathbf{X})(0)=\sum_{i=0}^{n-k}\frac{\partial X_i}{\partial x_i}(0)
\]

Now, we are in a position to prove the following geometric
property of the Kupka set.

\begin{proposition}\label{subcanonical} Let $\F\in\K_k(M,L)$ with a
  compact Kupka component $K(\F)=K$. Let $\jmath: K(\F) \hookrightarrow M$ be the
  inclusion map. Then
\begin{enumerate}
\item $\jmath:K(\F)\hookrightarrow M$ is subcanonicaly embedded.
\[
\wedge^{n-k+1}\nu(K(\F),M)=L|_{K(\F)},\qquad
\mathcal{K}_{K(\F)}=(\mathcal{K}_M\otimes L)|_{K(\F)}
\]

\item $c_1(\nu(K,M)=\jmath^{\ast}c_1(N_{\F}) \in H^2(K,\mathbb{Z})$.
\end{enumerate}
In particular, for a foliation on the projective space $\Po^n$,
$K$ is Fano if and only if $\dim(\F)
> \deg(\F).$
\end{proposition}

\begin{proof} Denote by $K=K(\F)$ .
The cocycle condition
  $\Om_{\alpha}=\lambda_{\alpha\beta}\Om_{\beta},$ implies
\[
d\Om_{\alpha}=d\lambda_{\alpha\beta}\wedge
\Om_{\beta}+\lambda_{\alpha\beta}d\Om_{\beta}.
\]
Therefore, $d\Om_{\alpha}=\lambda_{\alpha\beta}d\Om_{\beta}$ in
$K\cap U_{\alpha\beta}$. It follows that $\{d\Om_{\alpha}|_K\}$
defines a never vanishing holomorphic section of the line bundle
\[
\wedge^{n-k+1}\nu(K,M)^{\ast}\otimes L|_K.
\]

Hence, this line bundle is trivial. This implies that
\[
\wedge^{n-k+1}\nu(K,M)\simeq  L|_K=\det(N_{\F})|_K.
\]

From the exact sequence of vector bundles over the Kupka set
\[
\mathbf{0}\rightarrow TK\rightarrow TM|_K \rightarrow
\nu(K,M)\rightarrow\mathbf{0},
\]
after taking the dual sequence and the exterior power we get

\[
\mathcal{K}_{M}|_K=\wedge^{n-k+1}\nu(K,M)^{\ast}\otimes
\mathcal{K}_K \Rightarrow \mathcal{K}_K=(\mathcal{K}_M\otimes L)|_K
\]

It follows from the definition of the first Chern class:
\[
c_1(\nu(K,M))=c_1(\wedge^{n-k+1}\nu(K,M))= \jmath^{\ast}c_1(L)
=\jmath^{\ast}c_1(N_{\F}).
\]
\end{proof}

We proceed with another proof of the first item of the theorem
above:

\begin{proof}
 We first observe that $(T_{\F})|_K=TK$. In fact, on the
coordinate system $(U,\phi)$ given by Theorem (\ref{mede}), the
tangent sheaf of the foliation, is generated by the vectors fields
\begin{eqnarray*}
T_{\F}(U) & = &\mathcal{O}_M(U)\cdot
\left\langle\frac{\partial}{\partial y_1},\dots,
\frac{\partial}{\partial y_{k-1}}
,\mathbf{X}(x)\right\rangle \\
 TK|_U & = & \mathcal{O}_K\cdot\left\langle\frac{\partial}{\partial
y_1},\dots, \frac{\partial}{\partial y_{k-1}} \right\rangle
\end{eqnarray*}

Now, we have that $\mathbf{X}(0)\equiv0$, this implies that
\[
\mathcal{O}_K \cdot T_{\F}=\mathcal{O}_K\cdot
\left\langle\frac{\partial}{\partial y_1},\dots,
\frac{\partial}{\partial y_{k-1}} ,\mathbf{X}(0)\right\rangle=
\mathcal{O}_K\cdot\left\langle\frac{\partial}{\partial y_1},\dots,
\frac{\partial}{\partial y_{k-1}} \right\rangle
\]

Therefore, the canonical bundle of the Kupka set is the
restriction of the canonical bundle of the foliation, that means
\[
\mathcal{K}_K=(\mathcal{K}_{\F})|_{K(\F)}=(L\otimes
\mathcal{K}_M)|_K,\quad\mbox{and}\quad
\wedge^{n-k+1}\nu(K,M)=L|_K.
\]
The second part follows in the same way.\end{proof}

Also, in the same coordinate system $(U,\phi)$, the normal bundle
of the Kupka set is generated in $V=U\cap K$ by
\[
\nu(K,M)|_V=\mathcal{O}_K(V) \cdot\left\langle \frac{\partial}{\partial x_0},
\dots,\frac{\partial}{\partial x_{n-k}}\right\rangle \simeq \mathcal{O}_K^{n-k+1}(V),
\]

The normal sheaf, is coherent and it is not locally free at the
singular set. It has the same generators, but it has the relation
given by the transversal vector field $\mathbf{X}$, because it is
tangent to the foliation. That is
\[
N_{\F}(U)=\frac{\mathcal{O}_M(U)\cdot\left\langle
\frac{\partial}{\partial x_0}, \dots,\frac{\partial}{\partial
x_{n-k}}\right\rangle}{\sum X_i(x)\frac{\partial}{\partial x_i}}
\]
explicitly, we have the exact sequence
\[
\mathcal{O}_M(U)\stackrel{\cdot
\mathbf{X}}{\longrightarrow}\mathcal{O}_M^{n-k+1}(U)\longrightarrow
N_{\F}(U)
\]

\begin{example}
Here we have some examples and applications of the above result.

\begin{enumerate}
  \item \label{tcubic} The Proposition (\ref{subcanonical}) andy the Gherardelli
   theorem, that states that any subcanonical and projectively normal curve of
   $\Po^3$ is a complete intersection, we have that $\Gamma:\Po^1\to\Po^3$, the rational
   normal curve of degree $3$, parameterized by
     \[
       \Gamma(s:t)=(s^3:s^2t:st^2:t^3),
     \]
   can't be the Kupka set of any codimension one foliation on
   $\Po^3$.

  \item Any projective space linearly embedded, appears as a Kupka set of a degree zero
   foliation. Namely the foliation are the fibers of a linear
   projection $\Po^n\dashrightarrow \Po^{n-k}$ andy the Kupka set is the set of
   indetermination points.

   On the other hand, any projective space $\Po^{k-1}$ can not be the
   Kupka set of a foliation $\F$ of dimension $k$ and degree
   $\deg(\F)\geq k$.
\end{enumerate}
\end{example}

\subsection{The normal bundle of the Kupka set.}
Now, We to show another relevant property of the normal
bundle of a Kupka component.

Let $\mathbf{X}=X_1+\dots$ be the transversal vector field of a
compact Kupka component. The linear part $X_1(x)\neq0$. After a
linear change of coordinates, we may assume that
\begin{equation}
  \label{eq:Jordan}
X_1(x)
= \begin{pmatrix}
J_{\lambda_1}&      0        & \dots    & 0     \\
     0       & J_{\lambda_2} & \dots    & 0     \\
   \vdots    & \ddots        & \ddots   &\vdots \\
     0       &      0        & \dots    & J_{\lambda_r}
\end{pmatrix}\cdot
\begin{pmatrix}
x_1 \\ x_2 \\ \vdots \\ x_r
\end{pmatrix}
\quad x_j\in \C^{n_j},\quad \sum_{j=1}^r n_j=n-k+1.
\end{equation}

We assume that $J_{\lambda_j}$ are the Jordan blocs with
$\lambda_i\neq \lambda_j$, but eventually, $J_{\lambda_i}$ could have
several Jordan subblocs.

The Jordan canonical form, induces a decomposition in direct sum
of invariant subspaces
\[ \C^{n-k+1}=\bigoplus_{i=1}^r V(\lambda_i),\quad dim(V_{\lambda_i})=n_i.\]

\begin{theorem}\label{normal}
Let $\F\in \K_k(M,L)$ with transversal type
$\mathbf{X}=X_1+\cdots$, assume that the linear part $X_1$ is in
Jordan canonical form as in \ref{eq:Jordan}.  Then the normal
bundle decomposes as a direct sum of the type
\[
\nu(K,M) =\bigoplus_{i=1}^r E(\lambda_i), \mbox{ with }
rk(E_{\lambda_i})=n_i \] or it is indecomposable and projectively
flat.
\end{theorem}
\begin{proof}
Let $(U_{\alpha},\phi_{\alpha})$ be the charts of the Theorem
(\ref{mede}). The changes of coordinates
$\phi_{\alpha\beta}:=\phi_{\beta}\circ\phi_{\alpha}$ satisfy
$\phi_{\alpha\beta}^{\ast}\eta=\lambda\eta$ for some $\lambda\in
\mathcal{O}^{\ast}$ andy
$\eta=\imath_{\mathbf{X}}dx_0\wedge\dots\wedge dx_{n-k}$.

Put in coordinates
$\phi_{\alpha\beta}(x_{\beta},y_{\beta})=(x_{\alpha},y_{\alpha})$,
then for the linear part, we have
\[
\left(\frac{\partial \phi_{\alpha\beta} }{\partial x_{\beta}}\right)_{\ast}X_1=
\left(\frac{\partial \phi_{\alpha\beta} }{\partial
    x_{\beta}}\right)\cdot X_1 \cdot\left(
\frac{\partial \phi_{\alpha\beta} }{\partial x_{\beta}}\right)^{-1}= a
X_1\quad a\in\C^{\ast}.
\]
Then, the change of coordinates of the normal bundle preserve
the Jordan decomposition of the linear part.

Since a Jordan bloc has always an eigenvalue, the collection of
these eigenvalues provide a subbundle.

If the linear part is diagonal with only one eigenvalue, has not a
decomposition as a direct sum, but this case is projectively flat.
\end{proof}

\begin{remark}
The normal bundle
\[
\nu(K,M)=\bigoplus_{i=1}^r E_{\lambda_i}\Rightarrow c(\nu(K,M)=\sum
c_1(E_{\lambda_i})
\]

It would be interesting find relations between the Chern classes of
the vector bundles $E_{\lambda_i}$, we conjecture that
\[
\lambda_j n_i c_1(E_{\lambda_i}) = \lambda_i n_j
c_1(E_{\lambda_j})\quad n_i=\mbox{rank}(E(\lambda_i)),
\]
moreover, if $c_1(N_{\F})\neq0$ then
$J_{\lambda_i}=\lambda_i\mathbb{I}_{n_i\times n_i}$.

This holds for foliations of codimension 1 and 2 (\cite{kupka}).
\end{remark}
\begin{example}
Let $E \rightarrow M$ be an irreducible projectively flat vector
bundle. Let $L$ be a sufficiently holomorphic , ample line bundle
such that $E\otimes L$ has sufficiently many holomorphic sections.

Now, let $\sigma\in H^0(M,E\otimes L)$ be a section vanishing
transversely along a submanifold $K$. We consider the section
$\sigma$ as a meromorphic section on $\Po(E\otimes L)$, the
foliation is the pull--back of the flat connection form on
$\Po(E\otimes L)$.

\end{example}

Now, let $\F$ be a foliation with a compact Kupka component $K$.
Assume that the line bundle $L=\det(N_{\F})$ is ample and
transversal type of the Kupka component is the vector field
\[
\mathbf{X}(x_0,\dots,x_{n-k})=\sum_{i=0}^{n-k}(\lambda_i
x_{i}+\cdots )\frac{\partial}{\partial x_i}\qquad \lambda_i -
\lambda_j\neq0\neq \lambda_i \mbox{
  for all } i\neq j,
\]
in this case, we also have that $Div(\mathbf{X})_0=\sum
\lambda_i\neq0$.

In this situation, it is shown in \cite{kupka}, that the
eigendirections of the transversal vector field, defines
holomorphic line bundles $L_i$ over $K$ such that

\begin{equation}\label{chernnormal}
\nu(K,\Po^n)=\bigoplus_{i=0}^{n-k} L_i\Longrightarrow
c_1(\nu(K,\Po^n))=\sum c_1(L_i)=\jmath^{\ast}(c_1(N_{\F}))\neq0.
\end{equation}

On the other hand, the Chern classes $c_1(L_j)\in H^2
(K,\mathbb{Z})\quad j=0,\dots n-k$ satisfy the following
relations.

\begin{equation}\label{relchern}
\lambda_j\cdot c_1(L_i)- \lambda_i\cdot c_1(L_j)=0\in
H^2(K,\mathbb{Z}),\quad \forall\quad i,j=0,\dots n-k,
\end{equation}
it follows that the eigenvalues $(\lambda_0,\dots,\lambda_{n-k})$
of the linear part of the transversal type $\mathbf{X}$ may be
taken to be integers

A combination of the relations (\ref{chernnormal}) and
(\ref{relchern}) gives us that the Chern classes of the line
bundles $L_i$ are restriction of classes on the ambient manifold.
Namely

\[
\jmath^*(c_1(N_{\F}))=\sum_{i=0}^{n-k} c_1(L_i)=
  \left[\sum_{i=0}^{n-k}
    \left(\frac{\lambda_i}{\lambda_j}\right)\right]\cdot c_1(L_j),
\]
and then, we have
\begin{equation}
  \label{eq:chernclasses}
  H^2(K,\mathbb{Z})\ni
  c(L_i)=\frac{\lambda_i}{\sum_{i=0}^{n-k}\lambda_i} \jmath^{\ast}(c_1(N_{\F}))=
  \jmath^{\ast}\left(\frac{\lambda_i c_1(N_{\F})}{\sum_{i=0}^{n-k}\lambda_i}\right)
\end{equation}
We observe that the right side of the above equation, belongs to
$H^2(M,\mathbb{Q})$.

Now, assume that the line bundles $L_i$ are restriction of
holomorphic line bundles $\overline{L_i}\in Pic(M)$.

Then we, get a collection of $(n-k+1)$ integer classes
\[
c_i=\frac{\lambda_i c_1(N_{\F})}{\sum_{i=0}^{n-k}\lambda_i}\in
H^2(M, \mathbb{Z})\quad i=0,\dots,n-k
\]
such that,
\[
c_0 +\cdots +c_{n-k} = c_1(N_{\F}).
\]

We will find conditions for foliations on the projective
space, such that the line bundles on the Kupka set are
restrictions of line bundles on the projective space.

We give now other examples of submanifolds that can not
be the Kupka set of any foliation.

\begin{example}
Let
$\mathcal{S}_{p,q}:\Po^{p}\times\Po^q\hookrightarrow\Po^{(p+1)(q+1)-1}$
be the Segre embedding. We will see that for any
$c\in\mathbb{Z},$ the image
$\mathcal{S}_{d-1,1}(\Po^{d-1}\times\Po^1)=\Sigma_d\subset
\Po^{2d-1}$ cannot be the Kupka set of any foliation
$\F\in\F_{d+1}(2d-1,c)$.

We use the fact that the normal bundle
$\nu(\Sigma_d,\Po^{2d-1})$ is irreducible \cite{Bad}. Then, by the
Theorem (\ref{normal}), the normal bundle is projectively flat and
the linear transversal type must be radial. By the Poincar\'e
linearization theorem, the transversal type is actually radial.

The normal bundle $\nu(\Sigma_d,\Po^{2d-1})$ is projectively flat.
Then the bundle
\[ E=\nu(\Sigma_d,\Po^{2d-1})\otimes
\mathcal{O}_{\Sigma_d}(-c/d-1)\quad\mbox{has}\quad c(E)=1\in
H^{\ast}(\Sigma_d,\mathbb{Z})
\]
and by \cite[pages 114--115]{Kob}, it is flat. Since
$\Sigma_d$ is simply connected, the bundle $E$ is trivial. Therefore

\[
\nu(\Sigma_d,\Po^{2d-1})= E\otimes
\mathcal{O}_{\Sigma_d}(c/(d-1))=\bigoplus^{d-1}
\mathcal{O}_{\Sigma_d}(c/(d-1))
\]

In the codimension one case, i. e. $\Po^2\times\Po^1\simeq
\Sigma_3\subset\Po^5$, we have a direct proof without the Theorem
(\ref{normal}).

We will prove that, if $\Sigma_3\subset\Po^5$ is the Kupka set
of a codimension one holomorphic foliation, then its normal bundle
$\nu(\Sigma_3,\Po^5)\simeq L_0\oplus L_1$ for some holomorphic vector
bundles of rank one.

Recall that for codimension one foliations (see for instance
\cite{joseom}), the
transversal type is given by a linearizable and diagonal vector field

\[
\mathbf{X}(x_0,x_1)=\lambda_0 x_0\frac{\partial}{\partial
  x_0}+\lambda_1x_1\frac{\partial}{\partial x_1},
\]
where the eigenvalues are positive integers.

We have two cases:

First, if $\lambda_0\neq \lambda_1$ then $\nu(\Sigma_3,\Po^5)=L_0\oplus
L_1$ and the normal bundle splits.

If $\lambda_0=\lambda_1=1.$ In this case, we have that
$c_1(N_{\F})=c$ is even, since the degree of the Kupka set is
$\deg(K)=c^2/4$ (see \ref{degK}) and $\nu(\Sigma_3,\Po^5)$ is
projectively flat. Therefore the vector bundle
\[
E=\nu(\Sigma_3,\Po^5)\otimes \mathcal{O}_{\Sigma_3}(-c/2)
\]
has $c_1(E)=c_2(E)=0$. It follows from \cite{Kob} pag. 114--115,
that it is flat.

Since $\Sigma_3$ is simply connected, then
$E=\mathcal{O}_{\Sigma_3}\oplus\mathcal{O}_{\Sigma_3}$. Then
\[
\nu(\Sigma_3,\Po^5)=E\otimes\mathcal{O}_{\Sigma_3}(c/2)=
\mathcal{O}_{\Sigma_3}(c/2)\oplus\mathcal{O}_{\Sigma_3}(c/2).
\]

Again in this case the normal bundle splits.
\end{example}

The properties of the normal bundle imposses many restrictions on the
Kupka set.

We use the following result \cite{BP}.

\begin{theorem}\label{normsplit}
Let $X\subset\Po^4$ be a smooth surface. If the normal bundle is
decomposable i.e.  $\nu(X,\Po^4)=L_1\oplus L_2$ for some $L_1,L_2\in
Pic(X)$, then $X$ is a complete intersection.
\end{theorem}

With this result, we are able to prove:

\begin{theorem}
Set $\F\in \K_3(4,c)$. If the transversal type is not the radial
vector field or it is radial but $K$ is simply connected, then $K$
is a complete intersection and $\F$ has a meromorphic first
integral.
\end{theorem}
\begin{proof}
The transversal type is $\mathbf{X}_{pq}$ for some integers $1\leq p< q$ or $p=q=1$.
In the first case, the normal bundle splits in a sum of line bundles.

If $p=q=1$, then $\nu(K,\Po^4)$ is projectively flat and the foliation has even first Chern class andy
$\nu(K,\Po^4)(-c/2)$ is flat. Since $\pi_1(K,\ast)=1$, it is trivial and
$\nu(K,\Po^4)=\mathcal{O}_K(c/2)\oplus \mathcal{O}_K(c/2).$ Anyway, the conclusion follows from (\ref{normsplit}).
\end{proof}

Recall that in the codimension one case, for a compact, conected component of the Kupka set, there exists a
rank two holomorphic vector bundle $E$, with a section $\sigma$, and the exact sequence
\[
0\rightarrow\mathcal{O}\stackrel{\sigma}{\rightarrow}E\rightarrow\mathcal{J}_K(c)\rightarrow0.
\]
The Kupka set is a complete intersection if and only if $E$ splits \cite{CSoares}.

As a consequence of the Horrocks criteria for splitting of
holomorphic vector bundles on the projective space, we have

\begin{corollary}
Set $\F\in \K_{n-1}(n,c)\quad n\geq 5$, then $K$ is a complete intersection.
\end{corollary}
\begin{proof}
Let $\F\in \K_4(5,c)$ with Kupka set $K$. Then $K$ is simply
connected and by Lefschetz, it hyperplane section also is a simply
connected. Let $\ell:\Po^4\to \Po^5$ a linear embedding and
$\ell^{\ast}\F$ has a complete intersection Kupka set
$\ell^{-1}K$. Hence $K$ is a complete intersection.
\end{proof}

\begin{example}
It is shown in \cite{CSoares}, that the complex Torus
$\C^2/ \Lambda\subset \Po^4,$ associated to
the Horrocks--Munford bundle, can not be the Kupka set for a
foliations on $\Po^4$.

The Horrocks--Munford bundle is stable, but this property never holds for
the rank two holomorphic vector bundle associated to a Kupka component.
\end{example}

\subsection{Normal Forms}

Through this subsection, we assume that the
transversal type of a compact Kupka component is a linear vector
field of the type
\[
\mathbf{X}= \sum_{i=0}^{n-k} \lambda_i z_i\frac{\partial}{\partial z_i},
\quad \lambda_i\neq 0\neq \lambda_i-\lambda_j,\quad i,j=0\dots n-k
\]
and $\Om = \imath_{\mathbf{X}}dz_0\wedge \cdots\wedge dz_{n-k}$.

We consider such a vector field, as a vector in $\Lambda\in \C^{n+k+1},$
where $\Lambda=(\lambda_0,\dots ,\lambda_{n-k}).$

\begin{definition}
$\Lambda=(\lambda_0,\dots ,\lambda_{n-k})$ is \textit{ resonant }
if among the eigenvalues there exists integer relations of the
type
\[\lambda_s=\langle \mathbf{m},\Lambda\rangle=\sum_{j=0}^{n-k} m_j\lambda_j\]
where $\mathbf{m}=(m_0,\dots,m_{n-k}),\quad m_j\geq0,\quad \sum
m_j\geq2.$ Such a relation is called a \textit{resonance}. The
number $|\mathbf{m}|=\sum m_j$ is called the \textit{order} of the
resonance.
\end{definition}

Resonances are the obstructions for formal linearization of vector
fields of the type $\mathbf{Y}=\mathbf{X}+\cdots$. For us, an
interpretation of the meaning of the resonance is the following.

For a linear vector field $\mathbf{X}$, for all $j=0,\dots n-k$,
the hyperplanes $\{z_j=0\}$ are keeping invariant. If there is a
resonance of the type $\lambda_k=\langle
\mathbf{m},\Lambda\rangle$ the families of hypersurfaces $z_k+ c
z_0^{m_0}\cdots z_{n-k}^{m_{n-k}}= z_k+c Z^{\mathbf{m}}$ are
also invariants by $\mathbf{X}$. Namely
\[
\mathbf{X}(z_k+ c Z^{\mathbf{m}})= \lambda_k z_k+ c\mathbf{X}
(Z^{\mathbf{m}})=\lambda_k z_k +c \sum m_j\lambda_j Z^m=
\lambda_k(z_k + c Z^{\mathbf{m}}).
\]

Observe that the smooth hypersurfaces
$\{(z_k+cZ^{\mathbf{m}})=0\}$ and $\{z_k=0\}$ are tangent at the
origin.

The vector polynomials $Z^{\mathbf{m}}\mathbf{e}_j$ are resonant
if there is a resonance of the type $\lambda_j=\langle
\mathbf{m},\Lambda\rangle$. It follows that
\[
\mathbf{X} (Z^{\mathbf{m}}\mathbf{e}_j)=[\langle
\mathbf{m},\Lambda\rangle-\lambda_j] Z^{\mathbf{m}}\mathbf{e}_j
\]

In the sequel, we consider the forms
$\Om=\imath_{\mathbf{X}}dz_0\wedge\cdots \wedge dz_{n-k}$ with the
assumption that $(\lambda_0,\dots,\lambda_{n-k})$ are positive
integers.

We divide our analysis in three cases. The non--resonant, the
resonant and the radial as transversal type of a Kupka component.

\subsubsection{Non--resonant case}

\begin{theorem}
Let $K$ be a compact connected component of the Kupka set with transversal type
\[\mathbf{X}= \sum_{i=0}^{n-k} \lambda_i z_i\frac{\partial}{\partial z_i}\quad
\mbox{where}\quad \lambda_i\in\mathbb{N}
\]
and no--resonant. Then there exists
a neighborhood $K\subset U$ such that $\F|_U$ is defined by a
meromorphic closed $(n-k)$ form with poles along an invariant
divisor.
\end{theorem}
\begin{proof}
Consider $(U_{\alpha},\phi_{\alpha})$ a covering by charts of the
Kupka set andy
 $\Om_{\alpha}=\phi_{\alpha}^{\ast}\Om$.

In this case, we have that for all $j=0,\dots,n-k$, the
hyperplanes $\{z_j=0\}$ are the only smooth invariant
hypersurfaces, then we get
\[
\Om_{\alpha}=\sum_{j=0}^{n-k} \lambda_j z_{\alpha}^j dz_{\alpha}^{j+1}\wedge \cdots \wedge
dz_{\alpha}^{j-1}
\]
moreover, for all $j=0,\dots,n-k$, we have
$z_{\alpha}^j=g_{\alpha\beta}^jz_{\beta}^j,\quad\mbox{for
some}\quad g_{\alpha\beta}^j\in
\mathcal{O}^{\ast}(U_{\alpha\beta})$.

Consider the meromorphic closed forms
\[
\eta_{\alpha}=\frac{\Om_{\alpha}}{z_{\alpha}^0\cdots
z_{\alpha}^{n-k}}=
\frac{\lambda_{\alpha\beta}}{g^0_{\alpha\beta}\cdots
g^{n-k}_{\alpha\beta}} \frac{\Om_{\beta}}{z_{\beta}^0\cdots
z_{\beta}^{n-k}}=\frac{\lambda_{\alpha\beta}}{g^0_{\alpha\beta}\cdots
g^{n-k}_{\alpha\beta}}\eta_{\beta}.
\]

We have the following cocycles
\[c_{\alpha\beta}=\frac{\lambda_{\alpha\beta}}{g^0_{\alpha\beta}\cdots g^{n-k}_{\alpha\beta}}
\in \mathcal{O}^{\ast}(U_{\alpha\beta}).
\]

We claim that $c_{\alpha\beta}\equiv 1$. In fact, since
$\eta_{\alpha}$ and $\eta_{\beta}$ are closed andy
$\eta_{\alpha}=c_{\alpha\beta}\eta_{\beta}$, after taking exterior
derivative we have $0=dc_{\alpha\beta}\wedge\eta_{\beta}$. It
follows that $c_{\alpha\beta}$ is a holomorphic first integral of
the vector field $\mathbf{X}$, since the eigenvalues are positive,
any solution of the form accumulate at $0$, therefore
$c_{\alpha\beta}$ is a constant.

On the other hand, the forms $\eta_{\alpha}$ and $\eta_{\beta}$
have the same poles, it follows that the constant
$c_{\alpha\beta}=1$.

Therefore, $\eta_{\alpha}=\eta_{\beta}$ andy then, there exists a
meromorphic closed $(n-k)$ form $\eta$ defined in
$U=\cup U_{\alpha}\supset K$, with poles on an invariant divisor,
and representing the foliation.
This proof the non--resonant case.
\end{proof}

\subsubsection{Resonant Case}

Again, consider the diagonal vector field
\[
\mathbf{X}= \sum_{i=0}^{n-k} \lambda_i z_i\frac{\partial}{\partial
z_i},\quad \lambda_i\in \mathbb{N},
\]
and the $n-k$ form

\[
\Om= \sum_{j=0}^{n-k} (-1)^j\lambda_j x_j
dx_0\wedge \cdots \wedge\widehat{dx_j}\wedge\dots \wedge dx_{n-k}.
\]

Now, assume that the subset of eigenvalues
$\Lambda_{NR}=\{\lambda_0<\lambda_1<\dots <\lambda_{\ell}\}$ is
non--resonant and maximal with this property. It is always non
empty $\lambda_0\in\Lambda_{NR}$. The set
$\Lambda_{R}=\{\lambda_{\ell+1}<\dots < \lambda_{n-k}\}$ have
resonances, that always may be found of the type
\[
\lambda_{\ell +j} =
\lambda_0 m_0^{\ell +j}+\cdots +\lambda_{\ell} m_{\ell}^{{\ell}+j},\quad |\mathbf{m}|\geq2
\]
involving only non--resonant eigenvalues.

In order to see that, we observe first that $\lambda_{\ell+1}$ is
the smaller of all resonant eigenvalues andy it has an expression
involving only non--resonant eigenvalues.

We proceed by induction over $k$, with the index $\ell+k$ andy prove for $\ell+k+1$.
Assume that the statement is true for $\ell+k$ andy now, we have
\[
\lambda_{\ell + k+1}=\sum_{j=0}^{s}m_j^{k+1}\lambda_{j},
\]
since it involves only terms $\lambda_j<\lambda_{\ell+k+1}$. In
particular, in this sum only could appear resonant eigenvalues
$\lambda_{\ell+t}$ for $1\leq t\leq k.$ Then, these
eigenvalues may be replaced by the non--resonant terms.
Now, for any $s=1,\dots n-k-\ell$, consider all elements
\[
R(s)=\{ \textbf{m}\in \mathbb{N}^{\ell} \, |\,
\langle\textbf{m},\Lambda_{NR}\rangle=\lambda_{\ell+s}\}
\]
 And for any $\textbf{m}=(m_0,\dots,m_{\ell})\in R(s)$,
we consider the rational functions
\[
\varphi_{\textbf{m}}^{\ell+s} = \frac{x_{\ell+s}}{x_0^{m_0}\dots x_{\ell}^{m_{\ell}}}
=\frac{x_{\ell+s}}{x^{\mathbf{m}^s}}\quad
\lambda_{\ell+s}=m_0^s\lambda_0+\cdots + m_{\ell}^s\lambda_{\ell}
= \langle \mathbf{m}^s,\Lambda_{NR}\rangle.
\]

 And for any $\textbf{m}=(m_0,\dots,m_{\ell})\in R(s)$,
we consider the rational functions
\[
\varphi_{\textbf{m}}^{\ell+s} = \frac{x_{\ell+s}}{x_0^{m_0}\dots x_{\ell}^{m_{\ell}}}
=\frac{x_{\ell+s}}{x^{\mathbf{m}^s}}\quad
\lambda_{\ell+s}=m_0^s\lambda_0+\cdots + m_{\ell}^s\lambda_{\ell}
= \langle \mathbf{m}^s,\Lambda_{NR}\rangle.
\]

We will use the short notation. Define the monomial
\[
h^{\ell+s}_{\textbf{m}}(x_0,\dots,x_{\ell})=x_0^{m_0^s}\dots
x_{\ell}^{m_{\ell}^s},\quad \mathbf{m}=(m_0,\dots,m_{\ell})\in
R(s),
\]
then we have
\[
x_{\ell+s}=h^{\ell+s}_{\mathbf{m}} \cdot\varphi^{\ell+s}_{\mathbf{m}},\qquad
dx_{\ell+s}=\varphi_{\mathbf{m}}^s\cdot d h_{\mathbf{m}}^s +
h_{\mathbf{m}}^{\ell+s} \cdot d\varphi_{\mathbf{m}}^{\ell+s}
\]

The function $h^{\ell+s}_{\textbf{m}}$ is a resonant monomial associated to the
eigenvalue $\lambda_{\ell+s}$ then we have, for instance
$s=n-k-\ell$ that

\begin{eqnarray*}
\Om & = & \sum_{j=0}^{n-k} (-1)^i\lambda_i x_i
dx_0\wedge \dots \wedge\widehat{dx_i}\wedge\dots\wedge dx_{n-k} \\
    & = & \varphi^{n-k}_{\textbf{m}} \sum_{j=0}^{n-k-1} (-1)^i\lambda_i x_i
    dx_0\wedge \dots \widehat{dx_i}\dots\wedge dh^{n-k}_{\textbf{m}} \\
    & + & (-1)^{n-k}\lambda_{n-k}x_{n-k}dx_0\wedge\dots\wedge dx_{n-k-1} \\
    & + & h^{n-k}_{\textbf{m}}  \sum_{j=0}^{n-k-1} (-1)^i\lambda_i x_i
    dx_0\wedge \dots \wedge\widehat{dx_i}\wedge\dots\wedge d\varphi^{n-k}_{\textbf{m}}.
\end{eqnarray*}

Since $x_{n-k}=h^{n-k}_{\mathbf{m}}\varphi^{n-k}_{\mathbf{m}}$,
and the resonance implies that
\begin{eqnarray*}
&&\varphi^{n-k}_{\textbf{m}} \sum_{j=0}^{n-k-1} (-1)^i\lambda_i x_i
dx_0\wedge \dots \widehat{dx_i}\dots \wedge
dx_{n-k-1}\wedge dh^{n-k}_{\textbf{m}} \\
 && +
(-1)^{n-k}\lambda_{n-k}(\varphi_{\mathbf{m}}^{n-k}h^{n-k}_{\textbf{m}})
\,dx_0\wedge\dots\wedge dx_{n-k-1}=0.
\end{eqnarray*}

We finally we get
\begin{equation}\label{firstintegral}
   \Om =  h^k_{\mathbf{m}}\cdot\left(  \sum_{j=0}^{n-k-1} (-1)^ix_i
   dx_0\wedge \dots \widehat{dx_i}\dots dx_{n-k-1}\wedge d\varphi^{n-k}_{\mathbf{m}}\right)
\end{equation}

This expression may be written in the form
\[
\Om = h^{n-k}_{\mathbf{m}}\Om_1\wedge
d\varphi^{n-k}_{\mathbf{m}}\quad \Om_1
 =\sum_{j=0}^{n-k-1}(-1)^i \lambda_i x_i dx_0\wedge\dots\widehat{dx_i}\dots \wedge dx_{n-k-1}
\]

By this procedure, we are able to get inductively the expression
\begin{eqnarray*}
  \Om & = &  h^{\ell+1}_{\mathbf{m}_{\ell+1}}\dots h^{n-k}_{\mathbf{m}_{n-k}}
  \cdot \left( \Om_{\ell}\wedge d\varphi^{\ell+1}_{\mathbf{m}_{\ell+1}}\wedge\dots\wedge
  d\varphi^{n-k}_{\mathbf{m}_{n-k}} \right)   \\
   & = & h^{\ell+1}_{\mathbf{m}_{\ell+1}}\dots h^{n-k}_{\textbf{m}_{n-k}}\cdot
   \left(\Om_{\ell}\wedge
   d\left(\frac{x_{\ell+1}}{h^{\ell+1}_{\mathbf{m}_{\ell+1}}}\right)\wedge \dots \wedge
   d\left(\frac{x_{n-k}}{h^{n-k}_{\mathbf{m}_{n-k}}}\right)\right).
\end{eqnarray*}

Now, we define the non--resonant part as the $\ell$--form
\[
\Om_{NR}= \sum_{j=0}^{\ell} (-1)^j\lambda_j x_j
dx_0\wedge \cdots \wedge\widehat{dx_j}\wedge\dots \wedge dx_{\ell}
\]
and the \textit{logarithmic non--resonant part}
\[
\eta_{NR}=\frac{\Om_{NR}}{x_0\dots x_{\ell}}=
\sum_{j=0}^{\ell} (-1)^j\lambda_j \frac{dx_0}{x_0} \wedge \cdots \wedge\widehat{\frac{dx_j}{x_j}}\wedge\dots \wedge
\frac{dx_{\ell}}{x_{\ell}}.
\]

The above calculations implies the following theorem.

\begin{theorem}[Local Resonant Normal Form]\label{resonatFF}
Consider
\[
\Om= \sum_{j=0}^{n-k} (-1)^i \lambda_i x_i
dx_0\wedge\dots\widehat{dx_i}\dots \wedge dx_{n-k}, \quad
\lambda_i\in \mathbb{N},\quad \lambda_i\neq \lambda_j\mbox{ for }
i\neq j
\]
Let $\Lambda_{NR}=\{\lambda_0<\dots<\lambda_{\ell}\}$ and
$\Lambda_{R}=\{\lambda_{\ell+1}<\dots<\lambda_{n-k}\}$ be the non--resonant
part and the resonant part respectively.
Then there exists an integrating factor:
\[
\frac{\Om}{x_0\dots x_{\ell}H_{\textbf{M}}(x_0,\dots x_{\ell})}=
\eta_{NR}\wedge
d\varphi^{\ell+1}_{\textbf{m}_{\ell+1}}\wedge\dots\wedge
d\varphi^{n-k}_{\textbf{m}_{n-k}}.
\]
where $H_{\textbf{M}}=h^{\ell+1}_{\textbf{m}_{\ell+1}}\dots
h^{n-k}_{\textbf{m}_{n-k}}$, where
$\textbf{M}=(\textbf{m}_{\ell+1},\dots, \textbf{m}_{n-k})$ and
$\textbf{m}_{\ell+s}\in R(s)$.
\end{theorem}

This expression gives an integrability condition with meromorphic closed
forms. Observe that the function
\[
G_{\mathbf{M}}(x_0,\dots,x_{\ell})=x_0\dots
x_{\ell}H_{\mathbf{M}}(x_0,\dots,x_{\ell}),
\]
is a holomorphic integrating factor of $\Om$.

Also observe that if $\mathbf{m\neq n}\in R(s)$, the differential
of rational function
$\varphi_{\mathbf{m}}^{\ell+s}/\varphi_{\mathbf{n}}^{\ell+s}$
divides the form $\Om_{NR}$.

Now, we need to globalize this expression to a neighborhood of a
Kupka set with transversal type $\Om$.

We continue with the notation of the Theorem
(\ref{resonatFF}), but we fix (and omit) the
subindexes $\mathbf{m}\in R(s)$ and $\mathbf{M}$, for the
functions $h^{l+s},\, H,$ and $G$.

First, we observe that the pole of the logarithmic non--resonant
part, is invariant by the foliation. Then, for all
$j=0,\dots,\ell,\quad
\{x_{\alpha,j}=g_{\alpha\beta}^jx_{\beta,j}\}$. Also, we denote by
$\zeta_{\alpha}=(x_{\alpha,0},\dots,x_{\alpha,\ell})$, then:
\[
h^{\ell+k}(\zeta_{\alpha})=h^{\ell+k}_{\alpha}\qquad
H(\zeta_{\alpha})=H_{\alpha}\qquad G(\zeta_{\alpha})=G_{\alpha}
\]

Set $\gamma_{\alpha\beta}=(g_{\alpha\beta}^0,\dots
,g_{\alpha\beta}^\ell)$. If we consider $\gamma_{\alpha\beta}$ as
a diagonal matrix, then we have the equation
$\zeta_{\alpha}=\gamma_{\alpha\beta}\zeta_{\beta}.$ We also
consider $\gamma_{\alpha\beta}$ as a vector andy we get the
equations:
\[
  h^{\ell+k}(\gamma_{\alpha\beta}): =  h_{\alpha\beta}^{\ell+k} \qquad
  H(\gamma_{\alpha\beta}) : =  H_{\alpha\beta}
  \qquad G(\gamma_{\alpha\beta}) : = G_{\alpha\beta}.
\]
and with this convention, we get the equations
\[
h^{\ell+k}_{\alpha}=h^{\ell+k}_{\alpha\beta}\cdot h^{\ell+k}_{\beta}\qquad
H_{\alpha}=H_{\alpha\beta}\cdot H_{\beta}\qquad
G_{\alpha}=G_{\alpha\beta}\cdot G_{\beta}.
\]
Thus, we have
\[
\frac{\Om_{\alpha}}{G_{\alpha}}=\frac{\lambda_{\alpha\beta}}{G_{\alpha\beta}}
\cdot\frac{\Om_{\beta}}{G_{\beta}}.
\]

Since the forms are closed and in the Poincar\'e domain, it
follows that the never vanishing holomorphic function
\[
\frac{\lambda_{\alpha\beta}}{G_{\alpha\beta}}=c_{\alpha\beta}\in\C^{\ast},
\]
but unfortunately not necessarily $1$.
With the same notation, we can prove:

\begin{theorem}\label{intfact} Let $K$ be a compact, connected Kupka component
with transversal type $\Om$ as described above. Then. There exists a neighborhood
of the Kupka set such that the foliation is represented by a meromorphic $n-k$
form of the type.
\[
 \eta=
\eta_{NR} \wedge \theta_{\ell +1}\wedge
\dots \wedge \theta_{n-k},
\]
where $\theta_j,\quad j=1,\dots ,n-k-\ell$ are closed meromorphic 1--forms.
\end{theorem}
\begin{proof}
The proof will be by induction on the number of ressonant eigenvalues.

Let $\mathcal{U}=\{U_{\alpha}\}$ a covering of the theorem (\ref{mede}) andy write
$\Om_{\alpha}$ in the local normal form.
\[
\Om_{\alpha}=\eta_{\alpha}\wedge d\varphi_{\alpha}^{\ell+1}\wedge \cdots \wedge d\varphi_{\alpha}^{n-k},
\]
where $\eta_{\alpha}$ is logarithmic with poles on the divisor associated to the non--resonant eigenvalues.

Assume that $\Lambda_{NR}=\{\lambda_0,\dots,\lambda_{n-k-1}\}$. Then for each $\alpha,\beta$ we have
\[
\eta_{\alpha}\wedge d\varphi_{\alpha}= c_{\alpha\beta}\eta_{\beta}\wedge d\varphi_{\beta}\quad c_{\alpha\beta}\in\C^{\ast}.
\]

Since $\eta_{\alpha}=\eta_{\beta}$, we get
\[
\eta_{\alpha}\wedge (d\varphi_{\alpha}-c_{\alpha\beta}d\varphi_{\beta})=0.
\]

The form $\eta_{\alpha}$ has the division property, then we get
\[
d\varphi_{\alpha}=c_{\alpha\beta}d\varphi_{\beta}\Rightarrow \varphi_{\alpha}=c_{\alpha\beta}\varphi_{\beta}+b_{\alpha\beta}
\Rightarrow \frac{d\varphi_{\alpha}}{\varphi_{\alpha}}=\frac{d\varphi_{\beta}}{\varphi_{\beta}}
\]

Let  $\theta_{\alpha}=d\log\varphi_{\alpha}$, we get a meromorphic form
$\theta$ defined on the neighborhood $\displaystyle U=\bigcup_{\alpha} U_{\alpha}$ of $K$.  We get
\[
\frac{\Om}{G}=\eta_{NR}\wedge \theta.
\]

Now, we assume that the result holds, if the number of resonant eigenvalues is $s$ and prove the theorem for $s+1$.

The local normal form is
\[
\eta_{\alpha}\wedge d\varphi_{\alpha}^{\ell+1}\wedge\cdots \wedge d\varphi_{\alpha}^{\ell+s+1}.
\]
with $\eta_{\alpha}=\eta_{\beta}$ the logarithmic part. The form
\[
\eta_{\alpha}\wedge d\varphi_{\alpha}^{\ell+1}\wedge\cdots\wedge d\varphi_{\alpha}^{\ell+s+1}=
c_{\alpha\beta}\cdot \eta_{\beta}\wedge d\varphi_{\beta}^{\ell+1}\wedge\cdots\wedge d\varphi_{\beta}^{\ell+s+1}\quad
c_{\alpha\beta}\in\C^{\ast}.
\]
By the induction hypotheses, we are able to write
\[
\eta_{\alpha}\wedge d\varphi_{\alpha}^{\ell+1}\wedge\cdots\wedge d\varphi_{\alpha}^{\ell+s}=
a_{\alpha\beta}\cdot \eta_{\beta}\wedge d\varphi_{\beta}^{\ell+1}\wedge\cdots\wedge d\varphi_{\beta}^{\ell+s}
\quad  a_{\alpha\beta}\in\C^{\ast},
\]
we replaced by a form of the type
\[
\eta\wedge \theta_{\ell+1}\wedge \cdots \wedge \theta_{\ell+s}\qquad\mbox{with}\qquad d\theta_j=0\qquad \ell+1\leq j\leq \ell+s.
\]

Then, we have a new meromorphic local model
\[
\omega_{\alpha}=\Theta\wedge d\varphi_{\alpha}^{\ell+s+1}
\]
where $\Theta=\eta\wedge \theta_{\ell+1}\wedge \cdots \wedge \theta_{\ell+s}$ and
with $\eta$ a logarithmic $\ell$--form. Therefore
\[
\omega_{\alpha}=\Theta\wedge d\varphi^{l+s+1}_{\alpha}=
a_{\alpha\beta}'\cdot \Theta\wedge d\varphi^{l+s+1}_{\beta}\quad a_{\alpha\beta}'\in\C^{\ast},
\]
As in the case of a single resonat eigenvalue, we have
\[
\Theta\wedge (d\varphi_{\alpha}^{\ell+s+1}-a_{\alpha\beta}'\cdot d\varphi_{\beta}^{\ell+s+1})=0\Rightarrow
\varphi_{\alpha}^{\ell+s+1}=a_{\alpha\beta}'\varphi_{\beta}^{\ell+s+1}+b_{\alpha\beta}\quad b_{\alpha\beta}\in\C.
\]

Then $d\log(\varphi_{\alpha}^{\ell+s+1})=d\log(\varphi_{\beta}^{\ell+s+1})$ and we define the form
\[
(\theta_{\ell+s+1})|_{U_{\alpha}}=\frac{d\varphi_{\alpha}^{\ell+s+1}}{d\varphi_{\beta}^{\ell+s+1}}=
\frac{d\varphi_{\beta}^{\ell+s+1}}{d\varphi_{\beta}^{\ell+s+1}}=(\theta_{\ell+s+1})|_{U_{\alpha}}.
\]
And finally, we get $\eta=\eta_{NR}\wedge \theta_{\ell+1}\wedge \cdots \wedge \theta_{\ell+s+1}$ with $d\theta_j=0$.
\end{proof}

As a conclusion, we have show the following result.

\begin{theorem}\label{NFK}
Let $\F$ be a foliation with a Kupka component with transversal
type
\[
\mathbf{X}=\sum_{j=0}^{n-k} \lambda_i z_i \frac{\partial}{\partial
z_i},\qquad \lambda_j\neq \lambda_i,\quad \lambda_i\in \mathbb{N}.
\]
Then there exists a neighborhood $U$ of the Kupka set such that:
\begin{enumerate}
  \item There exists a meromorphic, closed and decomposable $n-k$
  form defining $\F$ outside and invariant divisor
  \item The reduced invariant divisor $\mathcal{D}_{red}$ is a normal
  crossing divisor and it is the pole of the logarithmic
  non--resonant part.
\end{enumerate}

\end{theorem}

\begin{example}
If the component has transversal type
\[
\mathbf{X}=\sum_{j=0}^{n-k}\lambda_j x_j\frac{\partial}{\partial
x_j} \quad \lambda_0=1,\quad \lambda_j<\lambda_{j+1}\in\mathbb{N}
\]
then $\Lambda_{NR}=\{\lambda_0\}$ andy the forma may be written as
\[
\Om= x_0^{N} d\left(\frac{x_1}{x_0^{\lambda_1}}\right)\wedge \dots \wedge
d\left(\frac{x_{n-k}}{x_0^{\lambda_{n-k}}}\right)\quad N=\sum_{j=0}^{n-k}\lambda_j
\]
\end{example}

\subsubsection{The radial case}
Now, we consider the case
\[
\mathbf{X}=\sum_{j=0}^{n-k} z_i\frac{\partial}{\partial z_i}
\]

Recall that a $(n-k)$ non singular foliation on a manifold $M$ has
\textit{a transversal projective structure}, if it is defined by
an Atlas of submersions $f_{\alpha}:U_{\alpha}\to \Po^{n-k}$ such
that $f_{\alpha}=\phi_{\alpha\beta}\circ f_{\beta}$ and
$\phi_{\alpha\beta}$ are restriction of projective transformations
on $\Po^{n-k}$.

The most simple example of transversally projective foliations, is
given by fibrations $F:M\to \Po^{n-k}$, another example are
generated by suspensions of a representation $\pi_1(M)\to
\mathbf{PGL}(n-k,\C)$.

The more general transversally projective foliations are a mixture
of these examples. It is given by the develop of the structure.

A singular holomorphic foliation has a projective transversal
structure if outside the singular set, there exists an invariant
divisor $D\subset M$ such that $\F|_{M-D-S(\F)}$ has a projective
transversal structure, see  \cite{Scardua}.


\begin{theorem}\label{RadLoc}
Let $K$ be a Kupka component with radial transversal type.
Then there exists a neighborhood $K\subset U$ such that $\F$
has a projective transversal structure.
\end{theorem}
\begin{proof}
Let $K$ be a compact connected Kupka
component of radial type andy let $\sigma:\widetilde{M}_K\rightarrow M$ be
the blow-up along $K$. Denote by $E:=\sigma^{-1}K$
the exceptional divisor. It is well known that $E=\Po(\nu(K,M))$, the
projectivization of the normal bundle of $K\subset M$.

Consider now the charts $\{U_\alpha,\phi_\alpha)\}$ of the theorem \ref{mede}, with coordinates
\[
\phi_{\alpha}: U_{\alpha}\rightarrow \C^{n-k+1}\times\C^{k-1}\quad \phi_{\alpha}(p)=(x_\alpha^0,x_\alpha^{1},\ldots,x_\alpha^{n-k},y_\alpha^{1},\ldots,y_\alpha^{k-1})
\]
of the adapted covering of $K$. The $(n-k)$--form $\Om_{\alpha}=\phi_{\alpha}^{*}(\eta)$, where
\[
\eta=\sum^{n-k}_{i=0}(-1)^{i}x_{i}dx^{0}\wedge\ldots\wedge\widehat{dx^{i}}
\wedge\ldots\wedge
dx^{n-k}
\]
as a vector field, the transversal type is
\[
\mathbf{X}=\sum^{n-k}_{i=0}x_{i}\frac{\partial}{\partial{x}_{i}}
\]

We consider the covering of the exceptional divisor
\[
U_{\alpha}\times \Po^{n-k}=\{(x_{\alpha},y_{\alpha},T):=(x_\alpha^0,\dots ,x_\alpha^{n-k},y_\alpha^1,\dots, y_\alpha^{k-1},[t^0:\cdots:t^{n-k}])\},
\]
and the blowing--up, in these local coordinates, is given by
\[
\widetilde{U_{\alpha}}=\{(x_\alpha,y_\alpha,T)\, |\,  x_\alpha^i t^j-x_\alpha^j t^i=0\quad \mbox{ for } i,j=0\dots ,n-k\}.
\]

Then, we have for example
\[
\sigma(x_\alpha^0,t^0,t^1,\ldots,t^{n-k},y_\alpha^{1},\ldots,y_\alpha^{k-1})=(x_\alpha^0,x_\alpha^0 t^1,x_{\alpha}^0t^2,
\ldots,x_\alpha^0t^{n-k},y_\alpha^{1},\ldots,y_\alpha^{k-1}),
\]
and then
\[
\sigma^{*}(\eta) = -(x_\alpha^0)^{n-k+1}\widetilde{\eta},\qquad
\widetilde{\eta}=dt^1\wedge \cdots\wedge dt^{n-k}
\]
where $\{x_\alpha^0=0\}$ respresented the local equation of the
exceptional divisor $E$ and $\widetilde{\eta}$ induces $\widetilde{\F}$, the strict
transformation of the foliation.

From this equation, we see that the strict transformed foliation $\widetilde{\F}$
is transversal to the exceptional divisor andy it defines the projective flat
connection on $\Po(\nu(K,M))$ andy the foliation
$\widetilde{\F}$ on the open set
\[
\widetilde{U}=\bigcup_{\alpha}\widetilde{U_{\alpha}}
\]
in non singular and defined by a representation
$\rho:\pi_{1}(K,\ast)\rightarrow \mathbf{PGL}(n-k,\C).$
So, it has a projective transverse structure.
\end{proof}

Transversally projective foliations, the structure is defined by the
Maurer--Cartan forms.
Therefore, on a neighborhood $U$ we have a family of meromorphic
1--forms for this structure.

Following \cite[pag- 132]{Kobayashi}, as a homogeneous space, the projective space $\Po^{n-k}$, is
$\mathcal{L}/\mathcal{L}_0$ where
\begin{eqnarray*}
\mathcal{L}   & = &\mathbf{PGL}(n-k,\C)=\mathbf{SL}(n-k+1,\C)/\mbox{center} \\
\mathcal{L}_0 & = & \left\{
\begin{pmatrix}
\mathbb{A}&0 \\ \mathbf{v} &a
\end{pmatrix}
\in \mathbf{SL}(n-k+1,\C)\right\}/\mbox{center},
\end{eqnarray*}
where $\mathbb{A}\in \mathbf{GL}(n-k,\C),\quad \mathbf{v}\in \C^{n-k}$ is a row $(n-k)$--vector,
$tr(\mathbb{A})=-a$,
\[
\mathcal{L}_1\left\{
\begin{pmatrix}
\mathbb{I}_{n}  & 0 \\ \mathbf{v} & 1
\end{pmatrix}
\| \mathbf{v}\in\C^n \mbox{ row vector}\right\}.
\]

The graded Lie Algebra $\mathfrak{l}=\mathfrak{g}_{-1}+\mathfrak{g}_0+\mathfrak{g}_1$
with this $\mathcal{L}/\mathcal{L}_0$, given by $\mathfrak{l}=\mathfrak{sl}(n+1,\C)$ andy
\[
\mathfrak{g}_{-1}=\left\lbrace \begin{pmatrix}
0 & \mathbf{v} \\ \mathbf{0} & 0
\end{pmatrix}   \right\rbrace
\quad
\mathfrak{g}_{0}=\left\lbrace \begin{pmatrix}
\mathbb{A} & \mathbf{0} \\ \mathbf{0} & a
\end{pmatrix}   \right\rbrace
\quad
\mathfrak{g}_{1}=\left\lbrace \begin{pmatrix}
\mathbf{0} & \mathbf{0} \\ \mathbf{v} & 0
\end{pmatrix}   \right\rbrace
\]

The invariant forms, are denoted by
\[
\{ \omega^i,\omega_i^j,\omega_j \| i,j= 1,\dots n-k \}
\]
and the structure equations are
\begin{eqnarray*}
   d\omega^i & = & -\sum \omega_k^i\wedge \omega^k   \\
 d\omega_j^i & = & -\sum \omega_k^i\wedge \omega_j^k-\omega^i\wedge \omega_j +
 \delta_j^i\sum \omega_k\wedge \omega^k      \\
   d\omega_j & = & -\sum \omega_k\wedge \omega_j^k,
\end{eqnarray*}

For the projective space $\Po^1$, the above equation looks like this
\begin{eqnarray*}
d\omega_{1} & = & \omega_{0} \wedge \omega^{1}, \\
d\omega_{0} & = & \omega^{1} \wedge \omega_{1}, \\
d\omega^{1} & = & \omega_{1} \wedge \omega_{0}.
\end{eqnarray*}

Therefore, on a neighborhood of the Kupka set, there is a family of meromorphic 1--forms
\[
\{\Om_i,\Om_{i}^{j},\Om^{j}\| i,j=1,\dots ,n-k\}
\]
obtained by the projection of the holomorphic 1--forms on the neighborhood of the exceptional divisor.

\vglue 10pt
\section{Baum-Bott theory.}
\label{sec:BB}

Let us recall the results of \cite{baum} and \cite{marcio} in a
general context. We need some preliminary notions.

A \textit{polynomial invariant function}  $F:M_{n\times n}(\C)\to \C$, is a function
which can be expressed as a complex polynomial in the entries of the
matrix, and satisfies
\[
F(\mathbb{A})=F(\mathbf{g}\cdot\mathbb{A}\cdot \mathbf{g}^{-1})
\]
for all $\mathbb{A}\in M_{n\times n}(\C)$ and $\mathbf{g}\in
GL(n,\C)$. It follows that $F$ define a function on $End(V)\to
\C$, for some complex vector space of dimension $n$.

The main example of a polynomial invariant functions, are the coefficients of
\[
 det(\mathbb{I}+t\mathbb{A} )=1+c_1(\mathbb{A})t+\cdots+c_n(\mathbb{A})t^n,
\]

Observe that $c_1(\mathbb{A})=tr(\mathbb{A})$ and $c_n(\mathbb{A})=det(\mathbb{A})$.

Now, Let
$\varphi\in\C[X_{1},\ldots,X_{n}]$ be a symmetric and
homogeneous of degree $n$. It defines an invariant function if it can be
expressed as
\[
\varphi(X_1,\dots,X_n)=\widetilde{\varphi}(c_1(X),\dots,c_n(X)),\quad
X=(X_1,\dots,X_n)
\]
for some polynomial $\widetilde{\varphi}$ and considering the
variables as the spectrum of a matrix.

Now, consider a 1 dimensional holomorphic foliation $\F$ on a compact complex manifold $M$ of complex dimansion $n$.
Such a foliation, is represented by a class $\mathbf{X}\in H^0(M,\Theta(L))$, where $L$ is a line bundle on $M$.

In this case, the tangent sheaf of $\F$ is the line bundle $L^{\ast}$ and the normal sheaf $N_{\F}$ may be considered as the
virtual bundle $TM-L$.

Assume that the singular set $S(\F)$ is a finite set of points. Now,
for any singular point $p\in S(\F)$, we take $z=(z_1,\dots,z_n)$ a coordinate
system around a singular point $p\in S(\F)$ such that $z(p)=0$, and
\[
\mathbf{X}(z)=\sum^{n}_{i=1}X_{i}(z)\frac{\partial}{\partial{z_{i}}},
\quad X_i(0)=0,
\quad J_{\mathbf{X}}(z)=\left(\frac{\partial X_i}{\partial z_j}(z)\right)
\]
be a holomorphic vector field representing $\F$ in a neighborhood
of $z(p)=0$.

The
Grothendieck point residue
\[
\mbox{Res}(\varphi(J_{\mathbf{X}}),\F,p)=\left(\frac{1}{2\pi
  i}\right)^n
\int_{\Gamma_{\epsilon}}\frac{\varphi(J_{\mathbf{X}})dz_{1}\wedge\ldots\wedge
  dz_{n}}{X_{1}(z)\ldots X_{n}(z)},
\]
where integration is over the $n$-cycle
\[
 \Gamma_{\epsilon}=\{z:|X_{i}(z)|=\varepsilon_{i}, 1\leq i\leq n \},\quad
 \epsilon=(\varepsilon_1,\ldots,\varepsilon_n)\in\mathbb{R}^{n}_{>0},\quad
 |\epsilon|<<1,
\]
oriented by declaring positive the form
$d(\ar{f_{1}})\wedge\ldots\wedge d(\ar{f_{n}})$.

It is not difficult see that this residue independent of the vector
field representing
$\F$ around $p$. We have the following.

\begin{theorem}[Baum-Bott \cite{baum}]\label{baum_teo}
Let $M$ be a compact complex manifold of complex dimension $n$. Let $\F$ be a holomorphic foliation by curves on $M$, represented by a meromorphic vector field $\mathbf{X}$. Then
\[
\varphi(N_{\F})[M]=\sum_{p\in  S(\F)}\mbox{Res}(\varphi(J_{X}),\F,p)
\]
\end{theorem}

For higher dimensional holomorphic foliations, there is also a
residue theorem: let $\F$ be a $k$-dimensional holomorphic
foliation on an $n$ dimensional complex manifold $M$. Let $Z$ be a
connected component of the singular set $S(\F)$. We consider the
sequence of homomorphism

\begin{align}
H_{j}(Z,\C)\overset{\jmath_{*}}\rightarrow
H_{j}(M,\C)\overset{\rho}\rightarrow
H^{2n-j}(M,\C)\qquad j=0,1,\ldots,2n,
\end{align}
where $\jmath_{*}$ is induced by inclusion and $\rho$ is the Poincar\'e
duality isomorphism. Set $\alpha_{*}=\rho\circ \jmath_{*}$.

\begin{theorem}[Baum-Bott \cite{baum}]\label{baum_teo2}
Let $\F$ be a foliation of $\dim(\F)=k$ on a manifold $M$ of
complex dimension $n$. Let
$\varphi\in\mathbb{C}[X_{1},\ldots,X_{n}]$ be symmetric and
homogenous of degree $r$, with $n-k<r\leq n$, and $Z$ be a
connected component of $S(\F)$.Then there exists a homology class
$\mbox{Res}(\varphi,\F,Z)\in H_{2n-2r}(Z,\mathbb{C})$ such that
\begin{enumerate}
\item $\mbox{Res}(\varphi,\F,Z)$ depends only on $\varphi$ and on the
  local behavior of the leaves of $\F$ near $Z$.
\item $\varphi(N_{\F})=\sum_{Z}\alpha_{*}\mbox{Res}(\varphi,\F,Z)$.
\end{enumerate}
\end{theorem}

We will use the above theorem in the case in that
$\deg(\varphi)=n-k+1$. The fact that
$\dim_{\C}S(\F)\leq k-1$, implies that only the
components of dimension $k-1$ of $S(\F)$ intervene. This is
because, since $\R(\varphi,\F,Z)\in  H_{2k-2}(Z,\C)$,
components of dimension smaller than $k-1$ contribute nothing.

\begin{theorem}\label{degK}
Let $\F\in \K_k(M,L)$ be a foliation with a compact connected
component $K$ and transversal type $\mathbf{X}$. Then
\[
[K]\cdot \mbox{Res}(c_1(J_{\mathbf{X}}))^{\cod(\F)+1},\mathbf{X},0)
=c_1(N_{\F})^{\mbox{cod}(\F)+1}.
\]
\end{theorem}
\begin{proof}
Suppose that $S_{k-1}(\F)=K(\F)$, where
$K(\F)$ compact connected Kupka set of transversal
type
\[
\mathbf{X}(x_0,\dots,x_{n-k})=
\sum^{n-k}_{i=0}X_i(x_0,\dots,x_{n-k})\frac{\partial}{\partial{x}_{i}},\quad
\cod(\F)=n-k.
\]

Set $\varphi=c^{n-k+1}_1(N_{\F})$. We have
\begin{equation}
\label{eq:equa1}
c^{n-k+1}_{1}(N_{\F})=\mbox{Res}(c^{n-k+1}_{1},\F,K(\F))[K(\F)],
\end{equation}
by Theorem \ref{baum_teo2}, but
\begin{eqnarray*}
 \mbox{Res}(c^{n-k+1}_{1},\F,K(\F))=
\left(\frac{1}{2\pi i}\right)^{n-k+1}
\int_{\Gamma_{\epsilon}}\frac{c^{n-k+1}_{1}(J_{\mathbf{X}}(0))dx_{0}
\wedge\ldots\wedge  dx_{n-k}}{X_0(x)\cdots X_{n-k}(x)},
\end{eqnarray*}
where integration is over the $n-k+1$-cycle
$\Gamma_{\epsilon}=\{z:|X_{i}(z)|=\varepsilon_{i}, 1\leq i\leq n \}$,
$\varepsilon=(\varepsilon_1,\ldots,\varepsilon_n)\in\mathbb{R}^{n}_{>0}$,
$|\varepsilon|<<1$.
\end{proof}

\vglue 10pt

\section{Holomorphic foliations on the projective space.}
\label{sec:Proj}

\subsection{The projective space.}
Recall the Euler sequence
\[
0\longrightarrow \mathcal{O}_n(-1) \longrightarrow
\bigoplus^{n+1}\mathcal{O}_n \longrightarrow
\Theta_n(-1)\longrightarrow 0,
\]
after taking the dual the $(n-k)$ exterior power, we get the exact
sequence
\[
0\longrightarrow
\Om^{n-k}_n(n-k)\longrightarrow
\bigwedge^{n-k}\left(\bigoplus^{n+1}\mathcal{O}_n\right)
\thickapprox\bigoplus^{\binom{n+1}{n-k}}\mathcal{O}_n
\longrightarrow \Om_n^{n-k-1}(n-k)\longrightarrow 0
\]

As a consequence, a global section $\Om$ of the sheaf
$\Om^{n-k}_{n}(c)$, may be thought as a polynomial $(n-k)$--form
on $\C^{n+1}$ with homogeneous coefficients of degree
$\deg_h(\Om)=c-n+k$, which we will still denote by $\Om$ and
satisfying
\[
\imath_{\mathcal{R}}\Om=0,\quad\text{ where }\quad
\mathcal{R}=x_{0}\frac{\partial}{\partial{x_{0}}}+
\dots+x_{n}\frac{\partial}{\partial{x_{n}}}
\]
is the radial vector field on $\C^{n+1}$.
Consequently, the space of sections $H^0(\Po^n,\Om_n^{n-k}(c)))$,
has dimension
\[
h^0(\Po^n,\Om_n^{n-k}(c))=\binom{c+k}{c}\cdot
\binom{c-1}{n-k},\quad c\geq n-k+1
\]
and $0$ otherwise. Therefore, only the classes $c\geq n-k+1$ could
carried foliations of dimension $k$.


Now, we introduce the notion of degree of a foliation
of the projective space. Let $\F$ be represented by a section
$\Om\in H^{0}(\Po^n,\Om^{n-k}_n(c)).$ We also assume that
$\dim_{\C}S(\F)\leq k-1$. Let $\ell:\Po^{n-k}\hookrightarrow\Po^n$
be a linear generic immersion with respect to $\F$, in particular,
$\ell(\Po^{n-k})$ does not hit the singular set. Now, consider the
pullback form $\ell^{*}(\Om)\in
H^{0}(\Po^{n-k},\Om^{n-k}_{n-k}(c))$. We observe that
$\Om^{n-k}_{n-k}(c)=\K_{n-k}(c),$ is a holomorphic line bundle, so
that $\K_{n-k}(c)=\mathcal{O}_{n-k}(d)$ for some $d\in\mathbb{Z}$.
Since $\K_{n-k} =\mathcal{O}_{n-k}(-(n-k+1))$, we get
$d=c-(n-k+1)$.

The zero divisor $\{\ell^{\ast}(\Om)=0\}$, reflects the set of
tangency between $\F$ and the linear subspace
$\ell(\Po^{n-k})\simeq\Po^{n-k}\subset\Po^n$. The \textit{degree}
of the foliation $\F$ denoted by $\deg(\F)$, is the degree of such
tangency divisor. By definition

\begin{equation}
  \label{eq:deg}
  \deg(\F)=c_1(N_{\F})-\cod(\F)-1=c_1(N_{\F})+\dim(\F)-n-1.
\end{equation}

Therefore, for foliations on the projective space, we have three
discrete invariants: The first Chern class of the normal bundle
$c_1(N_{\F})$ (or topological degree), the degree $\deg(\F)$ and
the homogeneous degree $\deg_h(\F)$ andy related by the formulaes
 \begin{eqnarray*}
     c_1(N_{\F}) & = & \deg(\F)+\cod(\F)+1, \\
   \deg_h(\F) & = & c_1(N_{\F})-\cod(\F)=\deg(\F)+1,
 \end{eqnarray*}

For instance, for any $n\geq2$ and for any codimension one
holomorphic foliation $\F$ on the projective space $\Po^n$, we
have the formula $c_1(N_{\F})=d_h(\F)+1=\deg(\F)+2$.

The example of a foliation $\F\in \F_k(n,n-k)$ and degree $\deg(\F)=0$ is
\[
\Om=\sum_{i=0}^{n-k} (-1)^ix_idx_0\wedge \cdots\wedge
\widehat{dx_i}\wedge \cdots\wedge dx_{n-k}=\imath_{\mathcal
R}dx_0\wedge\cdots \wedge dx_{n-k}
\]
whose leaves are the fibers of the linear fibration
\[
\Phi(z_0:\dots:z_n)=[z_0:\dots : z_{n-k}]:\Po^n\dashrightarrow\Po^{n-k}.
\]


It is well known that the singular set of a codimension one
foliation $\F$ in $\Po^n,$ has at least an irreducible component
in $S_{n-2}(\F)$ (see \cite{livro}). Recently, Corr\^ea Jr-
Fern\'andez P\'erez \cite{mauricio} proved that it is valid for
foliations of arbitrary dimension on projective manifolds.

\begin{theorem}\label{sing}
Let $\F\in\F_k(M,L)$ be a holomorphic foliation
such that
$\dim_{\C} S(\F)\leq  k-1$. Suppose that
$L$ is ample, then $S_{k-1}(\F)\neq\emptyset$
\end{theorem}

\begin{remark}
The hypotheses of the ampleness of the line bundle $L=\det(N_
{\F})$ is not restrictive. This hypothesis is always satisfy for
foliations on the projective space.
\end{remark}

\subsection{Foliations with a Kupka component}
We give the main examples of $k$ dimensional
holomorphic foliations $\F$ with a compact Kupka component on the
projective space.

\begin{example}\label{ex:1}
Let $\{f_0,\dots ,f_{n-k}\},\quad 1<k<n,$ be a collection of
homogenous polynomials  in $\C^{n+1}$ of degree $\deg(f_j)=d_j$.

We assume that $\{f:=f_0\cdots f_{n-k}=0\}$ is a
normal crossings divisor of degree $c=d_0+\dots +d_{n-k}$. In
other words, for all $j=0,\dots n-k,$ we have holomorphic sections
$f_j\in H^0(\Po^n,\mathcal{O}_n(d_j))$ andy $f\in
H^0(\Po^n,\mathcal{O}_n(c))$.

Consider the homogeneous $(n-k)$--form in $\C^{n+1}$
defined by

\begin{eqnarray*}
\Om(z_0,\dots,z_n)& = &\sum_{j=0}^{n-k} (-1)^j d_j\cdot f_j(z_0,\dots,z_n)
df_0\wedge\cdots\wedge\widehat{df_j} \wedge \cdots\wedge df_{n-k} \\
& = & (f_0\dots f_{n-k}) \, \sum_{j=0}^{n-k} (-1)^j d_j
\left\{\frac{df_0}{f_0}
\wedge\cdots\widehat{\frac{df_j}{f_j}}\cdots
\wedge\frac{df_{n-k}}{f_{n-k}}\right\}.
\end{eqnarray*}

The homogeneous $(n-k)$--form $\Om$ has degree $c-n+k$. It
vanishes the radial vector field $\mathcal{R}$. In fact, we
observe that $\Om=\imath_{\mathcal{R}}df_0\wedge \cdots df_{n-k}$,
the contraction by the radial vector field $\mathcal{R}.$
Moreover, the $(n-k)$--form $\Om$, defines a foliation $\F\in
\F_k(n,c),$ whose leaves are the fibers of the rational ramified
fibration

\begin{eqnarray*}
\Phi: \Po^{n} & \dashrightarrow & \Po^{n-k} \\
   z & \mapsto & [f_0^{m_0}(z):\cdots:f_{n-k}^{m_{n-k}}(z)]
\end{eqnarray*}
where $\{m_0,\dots ,m_{n-k}\}$ are integers relatively prime
satisfying the equations
\[
m_0d_0=\cdots =m_{n-k}d_{n-k}=d\quad
\mbox{for some}\quad d\in \mathbb{N}.
\]

The foliation belongs to the set $\K_k(n,c).$ The Kupka set is
the complete intersection $K=\{f_0=\cdots=f_{n-k}=0\}$. This is
also the indetermination locus of the rational map
$\Phi:\Po^n\dasharrow \Po^{n-k}$.

It is clear that the set $\{f_0=\cdots=f_{n-k}=0\}\subset S(\Om)$.
Now, consider any $x\in\{f_0=\cdots=f_{n-k}=0\}$ and a direct
calculation gives
\[
d\Om(x)=\left(\sum_{j=0}^{n-k}d_j\right)df_0\wedge\cdots\wedge
df_{n-k}(x) = c\,df_0\wedge\cdots\wedge df_{n-k}(x)\neq0,
\]
this last inequality, follows from the transversally condition.
The submanifold $K$ has degree $\deg(K)=d_0\cdots d_{n-k}.$

Finally, we observe that the transversal type of the Kupka component
is the linear vector field
\[
\mathbf{X}(x_0,\dots,x_{n-k})=
\sum_{j=0}^{n-k} d_j\cdot x_j\frac{\partial}{\partial x_j}.
\]
\end{example}

The foliations described above are called \textit{rational
  components}. They are irreducible components of the space of
foliations if $2\leq k<n$ \cite{CPV}. We denote
them  by
\[ \mathcal{R}_k(n:d_0,\dots,d_{n-k})\subset
\F_k(n,c),\quad\mbox{where}\quad c=\sum_{j=0}^{n-k}d_j.\]

Observe that the rational component has $n-k+1$ discrete
parameters, the degree of the divisors involved
$(d_0,\dots,d_{n-k})\in\mathbb{N}^{n-k+1}$, which are related with
the first Chern class of the normal sheaf of the foliation and the
degree of its Kupka component.

In the codimension one case, there is a complete classification of
the set $K_{n-1}(n,c)$ when $n\geq3.$ (see Brunella
\cite{brunella} and Calvo--Andrade \cite{joseom, omegar}).

\begin{theorem} The set
$\K_{n-1}(n,c,d)=\{\F\in \K_{n-1}(n,c)| \mbox{deg}(K)=d \}\neq
\emptyset$ if and only if the system of equations
\[
  x \cdot y = d \quad  x + y = c
\]
has positive integers as solution, say
$(a,b)\in\mathbb{N}\times\mathbb{N}$. In this case, we have
\[
\K_{n-1}(n,c,d)=\mathcal{R}_{n-1}(n:a,b).
\]
\end{theorem}
For any $\mathbb{Z}\ni c\geq 2$ the set $\K_{n-1}(n,c)$ has
$[c/2]$ irreducible components.
In the codimension one case, the transversal type is always of the
type
\[
\mathbf{X}_{pq}=px\frac{\partial}{\partial x}+qy\frac{\partial}{\partial
  y}\quad p,q\in\mathbb{N},\quad\mbox{where}
\quad \begin{cases} \mbox{radial} \quad &p=q=1\\ \mbox{non--resonant}\quad &(p,q)=1 \\ \mbox{resonant}\quad & 1=p<q
\end{cases}
\]
and we get $a=pc/p+q$ and $b=qc/p+q.$

An important intermediate problem, towards the classification of
the $k$ dimensional foliations with a Kupka component on the
projective space, is the following.

\begin{problem} Given $\F\in \K_k(n,c)$. Classify all the possible
transversal type of its Kupka component.
\end{problem}

\vglue 10pt

\section{Rational Fibrations. }
\label{sec:maintheorem}

\subsection{Geometrical Preliminary results}

We begin this section, with some results on the projective space.

There are two classes of theorems which are of our interest about submanifolds of a given manifold $M$:

Given $X\subset M$ a smooth submanifold and let $X\subset
U\subset M$ be an open neighborhood of $X$. The property that we need is
that any meromorphic object defined in $U$, may be extended to an
object on $M$.

The first results in this direction is the following.

\begin{theorem}[\cite{barth}, \cite{rossi}]\label{extension}
Let $Y$ be a connected analytic subset of $\Po^{n},\quad n\geq 2$,
with $\dim{Y}\geq 1$. Then any meromorphic function in a connected
neighborhood of $Y$ extends to a meromorphic function on all of
$\Po^{n}$.
\end{theorem}

On the other hand, given a smooth submanifold $X\subset \Po^n$,
the normal bundle $\nu(X,\Po^n)$ impose some conditions on the
embedding of $X \hookrightarrow \Po^n$. In this direction, there
are some results that we are going to use.

\begin{theorem}[\cite{Bad}]\label{NBSM}
Let $X\subset \Po^n$ be a smooth submanifold. Assume that
$\dim(X)=d\geq 3$ and $n=2d-1$. If the normal bundle
$\nu(X,\Po^n)$ decompose as a direct sum of line bundles, then
$Pic(X)=\Z\cdot [\mathcal{O}_X(1)]$
\end{theorem}

It follows from this result, that the normal bundle extends to
$\Po^n$ and $X$ is a complete intersection.

\begin{maintheorem}\label{maintheo}
Let $\F\in \K_k(n,c)$ be a holomorphic foliation with
$S_{k-1}(\F)=K(\F)$ and transversal type
\[
\mathbf{X}=\sum^{n-k}_{i=0}\lambda_{i}x_{i}\frac{\partial}{\partial{x}_{i}},
\quad\lambda_i\in\mathbb{N}\quad \lambda_{i}\neq \lambda_j, \mbox{
for all } i\neq j
\]
Then the leaves of the foliation are the fibers of a rational
fibration and it is represented by a $(n-k)$--form of the type
\[
\Om(z_0,\dots,z_n)=\sum^{n-k}_{i=0}(-1)^i \lambda_{i}\cdot
f_{i}(z_0,\dots,z_n) df_{0}
\wedge\ldots\wedge\widehat{df_{i}}\wedge\ldots\wedge df_{n-k}.
\]
\end{maintheorem}

\begin{proof} We use the Theorems \ref{intfact} and \ref{NFK} of the section \ref{sec:Kupka}.

Let $\Lambda_{NR}=\{\lambda_0<\cdots \lambda_{\ell}\}$ and
$\Lambda_R=\{\lambda_{\ell+1}<\dots <\lambda_{n-k}\}$ the resonant and non--resonant part respectively.
By (\ref{intfact}) and (\ref{NFK}), on a neighborhood of the Kupka set,
we have the integrating factor $G=G(x_0,\dots,x_{\ell}).$
It is a section on $U$ of the line bundle $\det(N_{\F})=\mathcal{O}(c)$.
Moreover, it can be extended to the projective space by Theorem \ref{extension}.

Now, the closed meromorphic forms $\theta_{\ell+j}$ are locally, logarithmic differentials of meromorphic functions.

Let $f_j\in H^0(\Po^n,\mathcal{O}(d_j))$ be the extension of the components
of the invariant divisor on $U.$ We also consider the extension of the rational functions
\[
\overline{ \varphi }^{\ell+k}=\frac{ f_{\ell+k} }{ f_0^{ m_0^{\ell+k} } \cdots f_{\ell}^{ m_{\ell}^{\ell+k} }  }
\]

Let $g_{\ell+k}=f_{\ell+k}+c_{\ell+k}\cdot f_0^{m_0^{\ell+k}}\cdots f_{\ell}^{m_{\ell}^{\ell+k}}$, where $c_{\ell+k}\in \C$.
Then the fibration is
\[
[f_0^{\lambda_0}:\dots:f_{\ell}^{\lambda_{\ell}}:g_{\ell+1}^{\lambda_{\ell+1}}:\dots
  :g_{n-k}^{\lambda_{n-k}}]:\Po^n\to \Po^{n-k}.
 \]
 \end{proof}

For the Corollary \ref{smallcod}, we begin with some remarks on the Kupka set.

Let $\F\in \K_k(M,L)$ with a Kupka component $K$, and transversal type
\[
\mathbf{X}(x_0,\dots,x_{n-k})=\sum_{i=0}^{n-k}(\lambda_ix_i+h.o.t)
\frac{\partial}{\partial x_i},
\]
then, the residue has the expression
\[
\mbox{Res}(c^{n-k+1}_{1},\F,K(\F))=\frac{(\sum^{n-k}_{i=0}
\lambda_{i})^{n-k+1}}{\prod^{n-k}_{i=0}\lambda_{i}}=
\prod_{j=0}^{n-k}\left(
\sum^{n-k}_{i=0}\frac{\lambda_{i}}{\lambda_j}\right).
\]

Then, by Theorem \ref{degK}, we have
\[
c_1(N_{\F})^{n-k+1}=
  \prod_{j=0}^{n-k}\left(
\sum^{n-k}_{i=0}\frac{\lambda_{i}}{\lambda_j}\right).
[K].
\]

Hence we have
\[
[K]=\prod_{i-0}^{n-k}
\left(\frac{\lambda_ic_1(N_{\F})}{\sum_{i=0}^{n-k}\lambda_i}\right)
\in H^{2(n-k+1)}(M,\C).
\]

In particular, for a foliation $\F\in\F_k(n,c)$, the above formula
looks as follows

\[
\deg(K)=\prod_{i=0}^{n-k}
\frac{\lambda_i( d(\F)+\cod(\F)+1)}{\sum_{i=0}^{n-k}\lambda_i}
=\prod_{i=0}^{n-k}\frac{\lambda_i c}{\sum_{i=0}^{n-k}\lambda_i}
\]

Now, Let $\F\in\K_k(n,c)$ be a foliation with a compact Kupka set
$K=K(\F)$. Let $\jmath:K\hookrightarrow\Po^n$ be the inclusion map and
assume the dimension assumption (\ref{eq:dimension}):
\[
4\leq \dim{\F}\qquad \dim(\F)\geq \cod(\F)+2.
\]

Since the linear part of the transversal type has different
eigenvalues, the normal bundle splits as a direct sum of line
bundles associated to the proper directions \cite{kupka}.
\[
\nu(K,\Po^n)=\bigoplus_{i=0}^{n-k} L_i.
\]

Because $K$ is subcanonical and $\nu(K,\Po^n)=\mathcal{O}_K(c),\quad c=c_1(N_{\F})$, we have
\[
c_1(\nu(K,\Po^n)=\jmath^{\ast}(c_1(N_{\F}))=
\sum_{i=0}^{n-k} c_1(L_i)\in H^2(K,\mathbb{Z}).
\]
On the other hand, the classes
\[\mathbf{c}_i:=\left(\frac{\lambda_i
  }{\sum_{i=0}^{n-k}\lambda_i}\right)\jmath^{\ast}(c_1(N_{\F}))\in
H^2(K,\mathbb{Q}),\quad i=0,\dots, n-k
\]
are precisely the Chern classes $c(L_i)$ of the line bundles
$L_i$, then these classes are integers and the same the numbers
\[
d_i=\frac{\lambda_i c}{\sum_{i=0}^{n-k}\lambda_i}\in \mathbb{Z}\simeq
H^2(K,\mathbb{Z}).
\]
 Then, we have proved.

\begin{ThmSmall}\label{smallcod}
Let $\F\in K_k(n,c)$ be a foliation of small codimension and
linear transversal type
\[
\mathbf{X}=\sum_{i=0}^k \lambda_i x_i \frac{\partial}{\partial
x_i}\quad \lambda_i\neq \lambda_j\quad \mbox{for all}\quad i\neq
j,
\]
then $\lambda_i\in \mathbb{N}$ and $\F$ is a rational fibration.
\end{ThmSmall}

Now, we consider the radial case.

\begin{Thm}\label{Radial}
Let $\F\in K_k(n,c)$ with a Kupka set $K$ and radial transversal
type. If $K$ is a complete intersection, then $\F$ is a rational
fibration of the type
\[[f_0:\dots:f_{n-k}]:\Po^n\to\Po^{n-k},\quad deg(f_j)=\frac{c}{n-k+1}\]
\end{Thm}
\begin{proof}
Let $U$ be a neighborhood of the Kupka set where the foliation has
a projective transverse structure. In this neighborhood we have
the Maurer--Cartan forms that defines the structure. We are able to
extend this forms and then the foliation has a projective
transversal structure on $\Po^n$.

Since $\Po^n$ is simply connected, any $\Po^{n-k}$ flat bundle is
$\Po^n\times \Po^{n-k}\stackrel{\pi_1}{\rightarrow} \Po^n$ and the rational
section is of the type $z\dashrightarrow(z,\Phi(z))$, then $\Phi$ must
be a rational fibration.
\end{proof}

\vglue 10pt
\section{An application: Normal form of non-integrable
  codimension one  distributions.}\label{sec:Noint}

Let $\omega$ be a germ of  $1$-form on $(\mathbb{C}^n,0)$.  We
denote by
\[ (d\omega)^s= \overbrace{d\omega \wedge \cdots \wedge d\omega}^{\mbox{s
times}}.
\]

Let $\omega\in H^0(M, \Om_M^1(L))$. We define the \emph{class of the
foliation} induced by $\omega$ to be the integer $r$ for which
generically
\[
\omega\wedge (d\omega)^{r-1}\neq 0, \quad  \omega\wedge
(d\omega)^{r}\equiv 0.
\]
See \cite{G} for the general theory.

The Kupka set of the distribution $\F$ induced by $\omega\in H^0(M,
\Om_M^1(L))$ is defined by
$$
K(\F)=\{p\in M;\ \omega(p)=0, \ (d\omega)^{r}(p)\neq 0 \}\subset
Sing(\F).
$$

The next result gives a Kupka type phenomena in the non integrable case.

\begin{theorem}[Kupka type phenomena]
Let $\omega\in H^0(M, \Om_M^1(L))$ a codimension one holomorphic
distribution of  class $r$. Given a connected Kupka component $K$
there exists a germ at $0\in \C^{2r}$ of a holomorphic $1$--form
\[
\sum_{i=1}^{2r}A_i(x_1,\dots,x_{2r})dx_i
\]
with an isolated singularity at $0\in \C^{2r}$, an open
covering $\{U_{\alpha}\}_{\alpha}$ of a neighborhood of $K\subset X$,
and a family of submersions
$\varphi_\alpha: U_\alpha\rightarrow \C^{2r}$ such that
\[
 \varphi_\alpha^{-1}(0)=K\cap U_\alpha\quad\mbox{and}\quad \omega_\alpha=\varphi_\alpha^{*}\eta.
\]
%

\end{theorem}

\begin{proof}
Since $(d\omega)^r(0)\neq 0$ it follows from the classical Darboux Theorem that
\[
d\omega= \sum_{i=1}^{r} dx_i\wedge dx_{i+r}.
\]

Then $d(\omega -\sum_{i=1}^{r} x_i dx_{i+r})=0$. From Poincar\'e
lemma we get that there exist $f\in \mathcal{O}_{n}$ such that
$\omega -\sum_{i=1}^{r} x_i dx_{i+r}=df.$ Using that
\[
(d\omega)^r=(-1)^{\frac{r(r-1)}{2}}r!dx_1\wedge \cdots \wedge
dx_{2r},
\]
we get
\[
0=\omega\wedge
(d\omega)^{r}=(-1)^{\frac{r(r-1)}{2}}r!\left(\sum_{i=1}^{r} x_i
dx_{i+r}+df\right)\wedge dx_1\wedge \cdots \wedge dx_{2r}.
\]

Then $df\wedge dx_1\wedge \cdots \wedge dx_{2r}=0$ implies that
\[
\frac{\partial f}{\partial x_i}=0 \quad \mbox{for all}\quad i= r+1,\dots ,n
\]
i.e., $f$ depends only of the variables
$(x_1,\dots,x_{2r})$.
\end{proof}
\vglue 8pt
\begin{example}
Let $\mathcal{D}$ be the distribution on $\Po^{n}$, of degree $d$, induced
given, in homogenous coordinates, by
\[
\omega=\sum_{i=1}^{r-1}(f_idf_{i+r}- f_{i+r}df_{i})
\]
such that $df_0\wedge \cdots \wedge df_{2r-1}(z)\neq 0$ .
The transversal type of $\omega$ is
\[
\omega=\sum_{i=1}^{r-1}(x_idx_{i+r}- x_{i+r}dx_{i}).
\]
In fact, the map
\[
(f_0,\dots,f_{2r-1}): \C^{n+1}\longrightarrow\C^{2r}
\]
is a submersion, since $df_1\wedge \cdots \wedge
df_{2r}(z)\neq 0$. Thus, there exist a holomorphic coordinate system
$(x_1,\dots,x_{2r},z)\in U\subset \C^{n+1} $ such that
$f_i|_{U}=x_i$.
\end{example}

\begin{nointthm}
Let $\mathcal{D}$ be a codimension one distribution on $\Po^n$, of class $r$,
given by a $1$--form $\omega \in H^0(\Po^n, \Om_{\Po^n}^1(d+2))$
such that $\cod Sing( (d\omega)^r)\ge 3$. Suppose that the Kupka
component of $\mathcal{D}$   has transversal type
\[
\sum_{i=0}^{r-1}(x_idx_{i+r}- x_{i+r}dx_{i}).
\]
Then $\mathcal{D}$ is induced, in homogenous coordinates, by the $1$-form
\[
\sum_{i=0}^{r-1}(f_idf_{i+r}- f_{i+r}df_{i}).
\]
\end{nointthm}

\begin{proof}
We have that the $(2r-1)$-form $\Theta=\omega\wedge (d\omega)^{r-1}$
is integrable and $K$ is a Kupka component  of radial type of
$\Theta.$ It follows from  Theorem \ref{Radial} that
\[
\Theta=\omega\wedge (d\omega)^{r-1}=\sum_{i=0}^{2r-1}(-1)^i f_idf_0
\wedge\dots \wedge  \widehat{df_i} \wedge \cdots\wedge df_{2r-1}.
\]
In particular, this implies that $(d\omega)^r$ is  globally
decomposable on $\mathbb{C}^{n+1}$. Since $\cod Sing( (d\omega)^r)\ge
3$ it follows from \cite{cerveau-darboux} that
\[
d\omega = \sum_{i=0}^{r-1}df_i\wedge df_{i+r}.
\]
Now, contracting by the radial vector field we conclude
\[
(d+2)\omega= i_R d\omega=\sum_{i=0}^{r-1}(m_{i+r}f_idf_{i+r}- m_if_{i+r}df_{i}),
\]
where $m_i=\deg(f_i)$. Comparing $\omega$ with its transversal type
we see that $m_i=m_j$, for all $i,j=0,\dots 2r-1.$
\end{proof}

\emph{ Acknowledgements.\/} We would to thank to the Department of
Algebra, geometry, topology and Analysis of the Valladolid
University, for its hospitality during the elaboration to this work.

The first author also thank to the Federal University of Minas Gerais UFMG
and IMPA, for the hospitality during the elaboration of this work.

O. Calvo--Andrade: Partially supported by CNPq

M. Corr\^ea Jr.: Partially supported by CAPES--DGU 247/11, CNPq 300352/2012--3 and
PPM--00169--13.

A. Ferrn\'andez--P\'erez: Partially supported by CNPq 301635/2013--7  and FAPEMIG APQ--00371--13

\enddocument